\documentclass[a4paper,11pt]{article}
\usepackage{amsmath,amsthm,amssymb,enumitem,xcolor}

\usepackage[hang]{footmisc}
\usepackage[nosort,nocompress,noadjust]{cite}

\usepackage[bookmarks=false,hyperfootnotes=false,colorlinks,
    linkcolor={red!60!black},
    citecolor={blue!50!black},
    urlcolor={blue!80!black}]{hyperref}

\renewcommand{\eqref}[1]{\hyperref[#1]{(\ref{#1})}}

\newlist{enumlist}{enumerate}{2}
\setlist[enumlist,1]{labelindent=0cm,label=\arabic*.,ref=\arabic*,labelwidth=2.5ex,labelsep=0.5ex,leftmargin=3ex,align=left,topsep=0.5ex,itemsep=1ex,parsep=1ex}
\setlist[enumlist,2]{labelindent=0cm,label=\theenumlisti.\arabic*.,ref=\arabic*,labelwidth=5ex,labelsep=0.5ex,leftmargin=5.5ex,align=left,topsep=0.5ex,itemsep=1ex,parsep=1ex}

\newlist{itemlist}{itemize}{1}
\setlist[itemlist]{labelindent=0cm,label=$\bullet$,labelwidth=2.5ex,labelsep=0.5ex,leftmargin=3ex,align=left,topsep=0.5ex,itemsep=1ex,parsep=1ex}

\numberwithin{equation}{section}

{\theoremstyle{definition}\newtheorem{definition}{Definition}[section]
\newtheorem*{definition*}{Definition}

\newtheorem*{example*}{Example}
\newtheorem*{examples*}{Examples}}

\newtheorem{proposition}[definition]{Proposition}
\newtheorem{lemma}[definition]{Lemma}
\newtheorem{theorem}[definition]{Theorem}

\newtheorem{letterthm}{Theorem}

{\theoremstyle{definition}}

\newcommand{\C}{\mathbb{C}}
\newcommand{\cC}{\mathcal{C}}
\newcommand{\eps}{\varepsilon}
\newcommand{\al}{\alpha}
\newcommand{\be}{\beta}
\newcommand{\ot}{\otimes}

\newcommand{\Z}{\mathbb{Z}}
\newcommand{\vphi}{\varphi}

\newcommand{\id}{\mathord{\text{\rm id}}}
\newcommand{\om}{\omega}
\newcommand{\N}{\mathbb{N}}
\newcommand{\ovt}{\mathbin{\overline{\otimes}}}

\newcommand{\Om}{\Omega}

\newcommand{\si}{\sigma}
\newcommand{\R}{\mathbb{R}}
\newcommand{\F}{\mathbb{F}}
\newcommand{\cH}{\mathcal{H}}

\newcommand{\cG}{\mathcal{G}}
\newcommand{\cK}{\mathcal{K}}

\newcommand{\cF}{\mathcal{F}}

\newcommand{\actson}{\curvearrowright}

\newcommand{\cU}{\mathcal{U}}
\newcommand{\Ker}{\operatorname{Ker}}

\newcommand{\Aut}{\operatorname{Aut}}

\newcommand{\cS}{\mathcal{S}}

\newcommand{\cL}{\mathcal{L}}

\newcommand{\mutil}{\widetilde{\mu}}
\newcommand{\Xtil}{\widetilde{X}}
\newcommand{\dpr}{^{\prime\prime}}

\newcommand{\otalg}{\mathbin{\otimes_{\text{\rm alg}}}}
\newcommand{\op}{^\text{\rm op}}
\newcommand{\End}{\operatorname{End}}
\newcommand{\Util}{\widetilde{U}}

\newcommand{\bim}[3]{\mathord{\raisebox{-0.4ex}[0ex][0ex]{\scriptsize $#1$}\hspace{0.2ex}\mbox{$#2$}\raisebox{-0.4ex}[0ex][0ex]{\scriptsize $#3$}}}

\begin{document}

\begin{center}
{\boldmath\LARGE\bf Measure equivalence embeddings of free groups\vspace{0.5ex}\\ and free group factors}

\bigskip

{\sc by Tey Berendschot\footnote{\noindent KU~Leuven, Department of Mathematics, Leuven (Belgium).\\ E-mails: tey.berendschot@kuleuven.be and stefaan.vaes@kuleuven.be.}\textsuperscript{,}\footnote{\noindent T.B. is supported by PhD fellowship fundamental research 1101322N of the Research Foundation Flanders.} and Stefaan Vaes\textsuperscript{1,}\footnote{S.V.\ is supported by Methusalem grant METH/21/03 –- long term structural funding of the Flemish Government and by FWO research project G090420N of the Research Foundation Flanders.}}
%
%\vspace{1ex}
%
%To appear in {\it Annales Scientifiques de l'Ecole Normale Sup\'{e}rieure}
\end{center}

\begin{abstract}\noindent
\noindent We give a simple and explicit proof that the free group $\F_2$ admits a measure equivalence embedding into any nonamenable locally compact, second countable group $G$. We use this to prove that every nonamenable locally compact, second countable group admits strongly ergodic actions of any possible Krieger type and admits nonamenable, weakly mixing actions with any prescribed flow of weights.\vspace{1ex}\\
We also introduce concepts of measure equivalence and measure equivalence embeddings for II$_1$ factors. We prove that a II$_1$ factor $M$ is nonamenable if and only if the free group factor $L(\F_2)$ admits a measure equivalence embedding into $M$. We prove stability of property~(T) and the Haagerup property under measure equivalence of II$_1$ factors.
\end{abstract}

\section{Introduction}

The von Neumann-Day problem famously asked whether every nonamenable group contains the free group $\F_2$ as a subgroup. Although this was shown to be false in \cite{Ols80}, the problem has a measurable group-theoretic counterpart that does hold. More precisely, it was shown by Gaboriau and Lyons in \cite[Theorem 1]{GL07} that for every countable nonamenable group $G$, the orbit equivalence relation of the Bernoulli action $G \actson [0,1]^G$ contains a.e.\ the orbits of an essentially free action $\F_2 \actson [0,1]^G$. This implies in particular that $\F_2$ is a measure equivalence (ME) subgroup of $G$ (in the sense of Gromov, see Definition \ref{def.ME-subgroup} for terminology). Gheysens and Monod showed in \cite[Theorem 1]{GM15} that actually every nonamenable locally compact second countable (lcsc) group $G$ contains $\F_2$ as an ME subgroup. Their proof combines the percolation theory methods \cite{GL07} with the solution of Hilbert's fifth problem on the structure of lcsc groups.

Our first main result is an explicit and simple construction to embed $\F_2$ as an ME subgroup into any nonamenable lcsc group. Because of the explicit nature of our measure equivalence embedding, we can derive several extra properties of weak mixing and strong ergodicity.

As we recall in Section \ref{subsec.ME}, a measure equivalence embedding of $\F_2$ into an lcsc group $G$ is defined by a nonsingular action $\F_2 \times G \actson (\Om,\gamma)$ whose restrictions to $\F_2$ and to $G$ admit a fundamental domain (in the appropriate sense when $G$ is nondiscrete) and such that $\F_2 \actson \Om/G$ admits an equivalent invariant probability measure (the finite covolume condition).

\begin{letterthm}\label{thm.main-ME-embedding}
Let $G$ be an lcsc group. Assume that $\nu$ is a probability measure on $G$ such that $\nu$ is equivalent with the left Haar measure and the convolution operator $\lambda(\nu)$ on $L^2(G)$ has norm less than $1/3$ (the existence of which is equivalent to the nonamenability of $G$). Consider the Bernoulli action $\F_2 \actson (X,\mu) = (G \times G,\nu \times \nu)^{\F_2}$.
\begin{enumlist}
\item The action $\F_2 \times G \actson (\Om,\gamma) = (X \times G,\mu \times \nu)$ given by
$$a \cdot (x,h) = (a \cdot x, (x_e)_1 h) \;\; , \;\; b \cdot (x,h) = (b \cdot x, (x_e)_2 h) \;\; , \;\; g \cdot (x,h) = (x,hg^{-1})$$
with $a,b$ freely generating $\F_2$ and $g \in G$, is a measure equivalence embedding of $\F_2$ into $G$.
\item This ME embedding is weakly mixing: the essentially free action $G \actson \F_2 \backslash \Om$ is weakly mixing.
\item This ME embedding is stably strongly ergodic: $G \actson \F_2 \backslash \Om$ is strongly ergodic and the diagonal product with any strongly ergodic pmp action $G \actson (Z,\zeta)$ remains strongly ergodic.
\end{enumlist}
\end{letterthm}

We prove Theorem \ref{thm.main-ME-embedding} in Section \ref{sec.proof-A}. Note that by Kesten's theorem (see \cite[Th\'{e}or\`{e}me 4]{BC73} in this generality), for any probability measure $\nu_0$ with full support on a nonamenable lcsc group, we have $\|\lambda(\nu_0)\| < 1$, so that a sufficiently high convolution power $\nu = \nu_0^{\ast k}$ satisfies $\|\lambda(\nu)\| < 1/3$. Therefore, Theorem \ref{thm.main-ME-embedding} proves very explicitly that $\F_2$ admits a measure equivalence embedding into any nonamenable lcsc group $G$. Note however that for countable groups $G$, we cannot recover in this way the full strength of the Gaboriau-Lyons theorem \cite{GL07}, who showed that the orbit equivalence relation of $G \actson [0,1]^G$ contains a.e.\ the orbits of an essentially free action of $\F_2$.

While the von Neumann-Day problem for groups was solved in the negative in \cite{Ols80}, the corresponding problem for II$_1$ factors remains wide open: does every nonamenable II$_1$ factor contain a copy of $L(\F_2)$ as a von Neumann subalgebra? While we cannot solve this problem, our simple approach to Theorem \ref{thm.main-ME-embedding} led us to a concept of \emph{measure equivalence (ME) embeddings for II$_1$ factors} and to proving that $L(\F_2)$ admits a measure equivalence embedding into any nonamenable II$_1$ factor.

Roughly speaking, an ME embedding of a II$_1$ factor $A$ into a II$_1$ factor $B$ is given by a Hilbert $A$-$B$-bimodule $\bim{A}{\cK}{B}$ together with the extra data (``the finite covolume condition'') of a \emph{finite} von Neumann algebra $Q \subset \End_{A\text{-}B}(\cK)$ such that the resulting bimodules $\bim{A}{\cK}{Q\op}$ and $\bim{Q}{\cK}{B}$ are both coarse and the latter is also finitely generated. If both are finitely generated, we get the concept of a \emph{measure equivalence} of $A$ and $B$. The precise definition is given in \ref{def.ME-II1}.

Our concept of ME embeddings for finite von Neumann algebras is compatible with the concept of ME embeddings for groups: we prove in Proposition \ref{prop.ME-groups} that an ME embedding of a countable group $\Gamma$ into a countable group $\Lambda$ gives rise, in a canonical way, to an ME embedding of $L(\Gamma)$ into $L(\Lambda)$.

On the one hand, our notion of measure equivalence (embedding) for II$_1$ factors is strong enough to capture qualitative properties of II$_1$ factors: if $A$ admits an ME embedding into $B$, then amenability or the Haagerup approximation property are inherited (see Proposition \ref{prop.inherited}) from $B$ to $A$, and Kazhdan's property~(T) is preserved under measure equivalence (Proposition \ref{prop.preserve-T}). On the other hand, our notion of measure equivalence embedding for II$_1$ factors is sufficiently flexible to obtain the following positive variant of the von Neumann-Day problem for II$_1$ factors.

\begin{letterthm}\label{thm.main-ME-embedding-II-1-factors}
A II$_1$ factor $M$ is nonamenable if and only if $L(\F_2)$ admits a measure equivalence embedding into $M$.
\end{letterthm}

Moreover, we prove in Proposition \ref{prop.counterexample} that in general, measure equivalence of II$_1$ factors is strictly weaker than commensurability (i.e.\ the existence of a finite index bimodule). Altogether these results show that our new concept of measure equivalence (embeddings) is quite natural.

Going back to groups, every ME embedding of an lcsc group $H$ into an lcsc group $G$ allows to induce nonsingular $H$-actions to nonsingular $G$-actions. In particular, the concrete ME embeddings of Theorem \ref{thm.main-ME-embedding} allow to induce nonsingular actions of the free group to nonsingular actions of an arbitrary nonamenable lcsc group $G$. Since the ME embeddings in Theorem \ref{thm.main-ME-embedding} have several extra properties, we can then prove at the end of Section \ref{sec.proof-C} the following result, whose third point answers a question posed by Danilenko in \cite{Dan21} because strongly ergodic actions on nonatomic spaces are in particular nonamenable.

\begin{letterthm}\label{thm.main-actions-nonamenable}
Let $G$ be a nonamenable lcsc group. Let $\R\actson (Z,\zeta)$ be an ergodic nonsingular flow and $\lambda\in (0,1]$.
\begin{enumlist}
\item There exists an essentially free, weakly mixing, nonsingular action $G\actson (X,\mu)$ that is nonamenable in the sense of Zimmer and such that the crossed product $L^{\infty}(X)\rtimes G$ has flow of weights $\R\actson Z$.
\item There exists an essentially free, weakly mixing, nonsingular action $G\actson (X,\mu)$ that is strongly ergodic and such that the crossed product $L^{\infty}(X)\rtimes G$ is of type III$_{\lambda}$.
\item There exists an essentially free, weakly mixing, nonsingular action $G\actson (X,\mu)$ that is strongly ergodic and such that there is no equivalent $G$-invariant $\si$-finite measure $\mu' \sim \mu$.
\end{enumlist}
\end{letterthm}

A more detailed formulation of this result can be found in Section \ref{sec.proof-C}.

\section{Preliminaries}

\subsection{Nonsingular actions and weak mixing}

Recall that a \emph{nonsingular} action of an lcsc group $G$ on a standard probability space $(X,\mu)$ is an action of $G$ on the set $X$ such that the map $G \times X \to X : (g,x) \mapsto g\cdot x$ is Borel and such that $\mu(g \cdot U) = 0$ whenever $U \subset X$ is Borel, $\mu(U) = 0$ and $g \in G$.

Recall that such a nonsingular action $G \actson (X,\mu)$ is called \emph{ergodic} if every $G$-invariant Borel set $U \subset X$ satisfies $\mu(U)=0$ or $\mu(X \setminus U) = 0$. Recall that $G \actson (X,\mu)$ is called \emph{weakly mixing} if for every ergodic, probability measure preserving (pmp) action $G \actson (Y,\eta)$, the product action $G \actson (X \times Y,\mu \times \eta)$ given by $g \cdot (x,y) = (g \cdot x,g \cdot y)$ is ergodic.

Finally recall that a pmp action $G \actson (X,\mu)$ is called \emph{mixing} if for all Borel sets $U, V \subset X$ and $\eps > 0$, the set $\bigl\{g \in G \bigm| |\mu(U \cap g \cdot V) - \mu(U) \, \mu(V)| \geq \eps \bigr\}$ is compact.

\subsection{Correspondences and measure equivalence embeddings}\label{subsec.ME}

Recall that a nonsingular action $H \actson (\Omega,\gamma)$ is dissipative and essentially free iff there exists a standard probability space $(Y,\eta)$ and a nonsingular isomorphism $\theta :  (H \times Y,\lambda \times \eta) \to (\Omega,\gamma)$, where $\lambda$ is a left Haar measure on $H$, such that for every $h \in H$, we have $\theta(hk,y) = h \cdot \theta(k,y)$ for a.e.\ $(k,y) \in H \times Y$. Then the quotient space $Y = H \backslash \Omega$ is well defined.

The notion of \emph{measure equivalence} for countable groups was introduced by Gromov in \cite{Gro91}. The following generalization to locally compact groups, as well as to measure equivalence embeddings, is taken from \cite[Definitions 3.1 and 3.5]{DL14}.

\begin{definition}\label{def.ME-subgroup}
Let $H$ and $G$ be lcsc groups. A \emph{correspondence} between $H$ and $G$ is a nonsingular action $H \times G \actson (\Omega,\gamma)$ on a standard probability space such that the actions $H\actson \Omega$ and $G\actson \Omega$ are both dissipative and essentially free.

Such a correspondence is said to be a \emph{measure equivalence embedding} of $H$ into $G$ if the nonsingular action $H \actson \Omega / G$ admits an equivalent $H$-invariant probability measure. It is said to be a \emph{measure equivalence} if both $\Omega / G$ and $H \backslash \Omega$ admit equivalent invariant probability measures.
\end{definition}

\subsection{Flow of weights and associated flow}\label{sec.flows}

Let $G$ be an lcsc group and $G \actson (Y,\eta)$ a nonsingular, ergodic, essentially free action. Denote by
$$\om : G \times Y \to \R : \om(g,y) = \log \frac{d(g^{-1} \cdot \eta)}{d\eta}(y)$$
the logarithm of the Radon-Nikodym cocycle and denote by $\Delta : G \to (0,+\infty)$ the modular function of $G$. Then the \emph{flow of weights} of the crossed product factor $L^\infty(Y) \rtimes G$ is given by the action of $\R$ by translation in the second variable on the ergodic decomposition of
$$G \actson Y \times \R : g \cdot (y,s) = (g \cdot y, \om(g,y) + \log \Delta(g) + s) \; .$$
So, if $G$ is nonunimodular, the flow of weights of $L^\infty(Y) \rtimes G$ need not be isomorphic to \emph{Krieger's associated flow} of $G \actson Y$, which is defined as the action of $\R$ by translation in the second variable on the ergodic decomposition of
$$G \actson Y \times \R : g \cdot (y,s) = (g \cdot y, \om(g,y) + s) \; .$$

In particular, whenever $G$ is nonunimodular and $G \actson (Y,\eta)$ is an essentially free action, pmp action such that $\Ker \Delta$ acts ergodically (e.g.\ because the action $G \actson (Y,\eta)$ is mixing), it follows that $L^\infty(Y) \rtimes G$ is always of type III: of type III$_1$ if $\Delta(G) \subset \R^*_+$ is dense and of type III$_\lambda$ if $\Delta(G) = \lambda^\Z$ and $\lambda \in (0,1)$. On the other hand, Krieger's associated flow for $G \actson (Y,\eta)$ is given by the translation action $\R \actson \R$.

Although for nonunimodular groups, the associated flow is not an orbit equivalence or von Neumann equivalence invariant, it is an interesting conjugacy invariant for the nonsingular action $G \actson (Y,\eta)$. In particular, if $G \actson (Y,\eta)$ is an ergodic nonsingular action, this associated flow is isomorphic with the translation flow $\R \actson \R$ if and only if there exists an equivalent $\si$-finite measure $\eta' \sim \eta$ that is $G$-invariant.

Note that if $H \times G \actson (\Omega,\gamma)$ is an essentially free, ergodic correspondence, then the actions $H \actson \Omega / G$ and $G \actson H \backslash \Omega$ and $H \times G \actson \Omega$ have isomorphic flows of weights.

On the other hand, writing $\Omega = X \times G$ where $X = \Omega / G$ and $\al : H \times X \to G$ is the $1$-cocycle such that the action $H \times G \actson \Omega = X \times G$ is given by $(h,g) \cdot (x,k) = (h \cdot x,\al(h,x) k g^{-1})$, the associated flow of $G \actson H \backslash \Omega$ is isomorphic with the action of $\R$ by translation in the second variable on the ergodic decomposition of
$$H \actson X \times \R : h \cdot (x,s) = (h \cdot x, \om_X(h,x) - \log \Delta_G(\al(h,x)) + \log \Delta_H(h) + s) \; ,$$
where $\om_X : H \times X \to \R$ is again the logarithm of the Radon-Nikodym cocycle.

\subsection{Strong ergodicity}

For completeness, we include a few well known results about strong ergodicity of nonsingular actions of lcsc groups, for which we could not find a good reference.

Recall that a nonsingular action $G \actson^\sigma (Y,\eta)$ of an lcsc group $G$ on a standard probability space $(Y,\eta)$ is called strongly ergodic if every bounded sequence $F_n \in L^\infty(Y)$ satisfying $\|F_n \circ \sigma_g - F_n \|_1 \to 0$ uniformly on compact subsets of $G$, satisfies $\|F_n - \eta(F_n)1\|_1 \to 0$. Since we are only dealing with bounded sequences, convergence in $\|\cdot\|_1$ is the same as convergence in $\|\cdot\|_2$ and is the same as convergence in measure.

\begin{lemma}\label{lem.criteria-strongly-ergodic}
Let $G \actson (Y,\eta)$ be a nonsingular action of an lcsc group $G$ on a standard probability space $(Y,\eta)$. Then the following statements are equivalent.
\begin{enumlist}[labelwidth=3ex,labelsep=0.5ex,leftmargin=3.5ex]
\item The action $G \actson (Y,\eta)$ is strongly ergodic.
\item If $F_n \in L^\infty(Y)$ is a bounded sequence and $F_n(g \cdot y) - F_n(y) \to 0$ for a.e.\ $(g,y) \in G \times Y$, then $\|F_n - \eta(F_n) 1\|_1 \to 0$.
\item[2'.] If $F_n \in L^\infty(Y)$ is a bounded sequence and $F_n(g \cdot y) - F_n(y) \to 0$ for a.e.\ $(g,y) \in G \times Y$, there exists a subsequence $(n_k)$ and a bounded sequence $t_k \in \C$ such that $F_{n_k}(y) - t_k \to 0$ for a.e.\ $y \in Y$.
\item If $(X,\mu)$ is a standard probability space, $F_n \in L^\infty(X \times Y)$ is a bounded sequence and $F_n(x,g\cdot y) - F_n(x,y) \to 0$ for a.e.\ $(g,x,y) \in G \times X \times Y$, then $\|F_n - (\id \ot \eta)(F_n) \ot 1\|_1 \to 0$.
\item[3'.] If $(X,\mu)$ is a standard probability space, $F_n \in L^\infty(X \times Y)$ is a bounded sequence and $F_n(x,g\cdot y) - F_n(x,y) \to 0$ for a.e.\ $(g,x,y) \in G \times X \times Y$, there exists a subsequence $(n_k)$ and a bounded sequence $H_k \in L^\infty(X)$ such that $F_{n_k}(x,y) - H_k(x) \to 0$ for a.e.\ $(x,y) \in X \times Y$.
\end{enumlist}
\end{lemma}

\begin{proof}
$1 \Rightarrow 2$.\ Assume that $G \actson^\si (Y,\eta)$ is strongly ergodic. Let $F_n \in L^\infty(Y)$ be a bounded sequence such that $F_n(g\cdot y) - F_n(y) \to 0$ for a.e.\ $(g,y) \in G \times Y$. Fix a compact neighborhood $K$ of $e$ in $G$. Fix a left Haar measure $\lambda$ on $G$. Define the bounded sequence $H_n \in L^\infty(Y)$ by
$$H_n(y) = \lambda(K)^{-1} \int_K F_n(h^{-1} \cdot y) \, d\lambda(h) \; .$$
For every $g \in G$, we have that
$$\|H_n \circ g - F_n\|_1 \leq \lambda(K)^{-1} \int_{K \times Y} |F_n(h^{-1}g \cdot y) - F_n(y)| \, d(\lambda \times \eta)(h,y) \; .$$
It follows from dominated convergence that $\|H_n \circ \si_g - F_n\|_1 \to 0$ for all $g \in G$. This holds in particular when $g = e$, so that $\|H_n - F_n\|_1 \to 0$ and $\|H_n \circ \si_g - H_n\|_1 \to 0$ for every $g \in G$. Note that
$$H_n(g \cdot y) = \lambda(K)^{-1} \int_{g^{-1}K} F_n(h^{-1} \cdot y) \, d\lambda(h) \; .$$
Taking $C > 0$ such that $\|F_n\|_\infty \leq C$ for all $n$, it follows that
$$\|H_n \circ \si_{g_1} - H_n \circ \si_{g_2}\|_1 \leq C \, \lambda(K)^{-1} \, \lambda(g_1^{-1} K \vartriangle g_2^{-1} K) \quad\text{for all $g_1,g_2 \in G$.}$$
We already proved that $\|H_n \circ \si_g - H_n\|_1 \to 0$ for every $g \in G$, so that $\|H_n \circ \si_g - H_n \|_1 \to 0$ uniformly on compact subsets of $G$. Since $G \actson (Y,\eta)$ is strongly ergodic, $\|H_n - \eta(H_n)1 \|_1 \to 0$. Since $\|H_n - F_n \|_1 \to 0$, we also get that $\|F_n - \eta(F_n)1\|_1 \to 0$.

$2 \Rightarrow 3$.\ Take a standard probability space $(X,\mu)$ and a bounded sequence $F_n \in L^\infty(X \times Y)$ such that $F_n(x,g\cdot y) - F_n(x,y) \to 0$ for a.e.\ $(g,x,y) \in G \times X \times Y$. By the Fubini theorem, we have that for a.e.\ $x \in X$, the sequence $F_n(x,\cdot)$ satisfies the assumptions of 2. Therefore, for a.e.\ $x \in X$, we have that $\|F_n(x,\cdot) - \eta(F_n(x,\cdot))1\|_1 \to 0$. Integrating over $x \in X$, it follows from dominated convergence that $\|F_n - (\id \ot \eta)(F_n) \ot 1\|_1 \to 0$.

$3 \Rightarrow 2$ is trivial by taking for $X$ a one point set.

$2 \Rightarrow 1$. Assume that $G \actson^\si (Y,\eta)$ is not strongly ergodic. We then find a bounded sequence $F_n \in L^\infty(Y)$ and a $\delta > 0$ such that $\|F_n \circ \si_g - F_n\|_1 \to 0$ uniformly on compact subsets of $G$ and $\|F_n - \eta(F_n)1\|_1 \geq \delta$ for all $n \in \N$. Choose a probability measure $\nu$ on $G$ that is equivalent with the Haar measure. Then,
$$\int_{G \times Y} |F_n(g \cdot y) - F_n(y)| \, d(\nu \times \eta)(g,y) = \int_G \|F_n \circ \si_g - F_n \|_1 \, d\nu(g) \to 0 \; .$$
We can thus choose a subsequence $(n_k)$ such that $F_{n_k}(g \cdot y) - F_{n_k}(y) \to 0$ for a.e.\ $(g,y) \in G \times Y$. Since $\|F_{n_k} - \eta(F_{n_k})1\|_1 \geq \delta$ for all $k$, we conclude that $2$ does not hold.

$2 \Leftrightarrow 2'$ and $3 \Leftrightarrow 3'$ follow immediately because any sequence converging to $0$ in measure admits a subsequence converging to $0$ a.e.
\end{proof}

The following lemma was already proven in e.g.\ \cite[Lemma 2.5]{Ioa14}. Using the previous lemma, we can give a very short argument, which we include for completeness.

\begin{lemma}\label{lem.strongly-ergodic-translation-action}
Let $G$ be an lcsc group and let $\eta$ be a probability measure on $G$ that is equivalent with the Haar measure. The translation action $G \actson (G,\eta)$ is strongly ergodic.
\end{lemma}
\begin{proof}
Let $F_n \in L^\infty(G)$ be a bounded sequence such that $F_n(gh) - F_n(h) \to 0$ for a.e.\ $(g,h) \in G \times G$. Since the map $(g,h) \mapsto (gh,h)$ is nonsingular, also $F_n(g) - F_n(h) \to 0$ for a.e.\ $(g,h) \in G \times G$. By the Fubini theorem, we can choose $h \in G$ such that $t_n = F_n(h)$ is a bounded sequence and $F_n(g) - t_n \to 0$ for a.e.\ $g \in G$. By Lemma \ref{lem.criteria-strongly-ergodic}, the action $G \actson (G,\eta)$ is strongly ergodic.
\end{proof}

\begin{proposition}\label{prop.stability-of-strong-ergodicity-under-correspondence}
Let $H$ and $G$ be lcsc groups and $H \times G \actson (\Om,\gamma)$ a correspondence between $H$ and $G$ (see Definition \ref{def.ME-subgroup}). Then $H \actson \Om/G$ is strongly ergodic iff $G \actson H \backslash \Om$ is strongly ergodic.
\end{proposition}
\begin{proof}
By symmetry, it suffices to prove that $G \actson H \backslash \Omega$ is strongly ergodic assuming that $H \actson \Om / G$ is strongly ergodic.

We view $H$ as acting on the left and $G$ as acting on the right. We identify $\Om = X \times G$ in such a way that the right $G$-action is given by right translation in the second variable. We then identify $\Om / G = X$. Let $F_n \in L^\infty(H \backslash \Omega)$ be a bounded sequence such that $F_n(y \cdot g) - F_n(y) \to 0$ for a.e.\ $(g,y) \in G \times H \backslash \Omega$. We view $F_n$ as a bounded $H$-invariant sequence in $L^\infty(X \times G)$. Then, $F_n(x,kg) - F_n(x,k) \to 0$ for a.e.\ $(g,x,k) \in G \times X \times G$. By Lemmas \ref{lem.criteria-strongly-ergodic} and \ref{lem.strongly-ergodic-translation-action}, after passage to a subsequence, we find a bounded sequence $S_n \in L^\infty(X)$ such that $F_n(x,k) - S_n(x) \to 0$ for a.e.\ $(x,k) \in X \times G$.

Since $F_n$ is $H$-invariant, it follows that $S_n(h \cdot x) - S_n(x) \to 0$ for a.e.\ $(h,x) \in H \times X$. By strong ergodicity of $H \actson X$, after a further passage to a subsequence, we find that $S_n(x) - t_n \to 0$ for a.e.\ $x \in X$, where $t_n \in \C$ is a bounded sequence. So we have proven that $F_n(y) - t_n \to 0$ for a.e.\ $y \in H \backslash \Omega$. By Lemma \ref{lem.criteria-strongly-ergodic}, the strong ergodicity of $G \actson H \backslash \Omega$ follows.
\end{proof}

\section{Measure equivalence embeddings of the free groups}

In this section, we prove Theorem \ref{thm.main-ME-embedding}, as well as the following variant.

\begin{theorem}\label{thm.ME-embedding-nonamenable}
Let $G$ be an lcsc group. Then the following statements are equivalent.
\begin{enumlist}
\item $G$ is nonamenable.
\item There exists a measure equivalence embedding of $\F_2$ into $G$.
\item There exists a measure equivalence embedding $\F_2 \times G \actson \Omega$ of $\F_2$ into $G$ such that the associated action $G \actson Y = \F_2 \backslash \Omega$ is essentially free, weakly mixing, strongly ergodic and stably strongly ergodic: for every strongly ergodic pmp action $G \actson (Z,\zeta)$, the diagonal action $G \actson Y \times Z$ remains strongly ergodic.
\end{enumlist}
\end{theorem}

So nonamenability of an lcsc group $G$ is precisely characterized by the existence of a measure equivalence embedding of $\F_2$. Similarly, noncompactness of $G$ is precisely characterized by the existence of a measure equivalence embedding of $\Z$. Moreover, this ME embedding can be chosen to be weakly mixing. For countable groups $G$, this was proven in \cite[Theorem 6.1]{BN13}. For completeness, we include a proof in the general locally compact setting.

\begin{theorem}\label{thm.ME-embedding-noncompact}
Let $G$ be an lcsc group. Then the following statements are equivalent.
\begin{enumlist}
\item $G$ is noncompact.
\item There exists a measure equivalence embedding of $\Z$ into $G$.
\item There exists a measure equivalence embedding $\Z \times G \actson \Omega$ of $\Z$ into $G$ such that the associated action $G \actson Y = \Z \backslash \Omega$ is essentially free and weakly mixing.
\end{enumlist}
\end{theorem}

All our measure equivalence embeddings arise from the following canonical construction of $1$-cocycles for the Bernoulli action of a free group. For the group of integers, this construction was introduced in \cite{BN13}. Let $(X_0,\mu_0)$ be a standard probability space and $\Gamma \cong \F_k$ the free group with $1 \leq k < \infty$ free generators $a_1,\ldots,a_k$. Let $\cG$ be a Polish group and, for every $i \in \{1,\ldots,k\}$, $\vphi_i : X_0 \to \cG$ a Borel map. Then we consider the unique $1$-cocycle for the Bernoulli action $\Gamma \actson (X,\mu) = (X_0,\mu_0)^\Gamma$ given by
\begin{equation}\label{eq.my-1-cocycle}
\al : \Gamma \times X \to \cG : \al(a_i,x) = \vphi_i(x_e) \quad\text{for all $i \in \{1,\ldots,k\}$ and $x \in X$.}
\end{equation}

\subsection{Random ergodic theorems}

The following lemma is Kakutani's random ergodic theorem \cite[Theorem 3]{Kak50}, for which we provide a simple proof.

\begin{lemma}\label{lem.random-ergodic}
Let $\cH$ be a separable Hilbert space, $(X_0,\mu_0)$ a standard probability space and $\vphi : X_0 \to \cU(\cH)$ a Borel map. Consider the Bernoulli shift $S$ on $(X,\mu) = (X_0,\mu_0)^\Z$ given by $S(x)_n = x_{n-1}$. Define the unitary $V \in \cU(L^2(X,\mu) \ot \cH)$ by $(V^* \xi)(x) = \vphi(x_0)^* \xi(S(x))$.

Let $G_0 \subset \cU(\cH)$ be the subgroup generated by the essential range of $\vphi$. Then, the $V$-invariant vectors are given by $(L^2(X,\mu) \ot \cH)^V = 1 \ot \cH^{G_0}$.
\end{lemma}

\begin{proof}
We identity $L^2(X,\mu) \ot \cH = L^2(X,\cH)$. Let $\xi \in L^2(X,\cH)$ be a $V$-invariant vector. It suffices to prove that the function $\xi : X \to \cH$ is essentially constant. Consider the Bernoulli action $\Z \actson (X,\mu)$ and the $1$-cocycle $\al : \Z \times X \to \cU(\cH)$ given by \eqref{eq.my-1-cocycle}. Then, $\xi(n \cdot x) = \al(n,x)(\xi(x))$ for all $n \in \Z$ and a.e.\ $x \in X$.

Define the index set $I = \Z \sqcup \{\ast\}$ by adding one element to $\Z$. Fix $k \in \Z$ and define the injective maps $\theta_k : \Z \to I$ and $\theta'_k : \Z \to I$ by $\theta_k(n) = \theta'_k(n) = n$ when $n \in \Z \setminus \{k\}$, while $\theta_k(k) = k$ and $\theta'_k(k) = \ast$. We write $(\Xtil,\mutil) = (X_0,\mu_0)^I$ with the measure preserving factor maps
$$\pi_k : \Xtil \to X : (\pi_k(x))_n = x_{\theta_k(n)} \quad\text{and}\quad \pi'_k : \Xtil \to X : (\pi'_k(x))_n = x_{\theta'_k(n)} \; .$$
We prove that $\xi(\pi_k(x)) = \xi(\pi'_k(x))$ for a.e.\ $x \in \Xtil$. This means that the map $\xi$ is essentially independent of the coordinate $x_k$. Since this then holds for all $k \in \Z$, we will have proven that $\xi$ is essentially constant.

For every $\zeta \in L^2(X,\cH)$, we define $\zeta_{n,k} \in L^2(\Xtil,\cH) : \zeta_{n,k}(x) = \zeta(n \cdot \pi_k(x))$ and we similarly define $\zeta'_{n,k}$. The maps $\zeta \mapsto \zeta_{n,k}$ and $\zeta \mapsto \zeta'_{n,k}$ are isometric. When $\zeta$ depends on only finitely many coordinates $x_i$, we have that $\zeta_{n,k} = \zeta'_{n,k}$ when $|n|$ is sufficiently large. Such functions $\zeta$ form a dense subspace of $L^2(X,\cH)$ and it follows that
\begin{equation}\label{eq.good-limit}
\lim_{|n| \to \infty} \|\zeta_{n,k} - \zeta'_{n,k}\|_2 = 0 \quad\text{for all $\zeta \in L^2(X,\cH)$ and $k \in \Z$.}
\end{equation}
First assume that $k \geq 1$. For all $n \geq 1$, the map $x \mapsto \al(n,x)$ only depends on the coordinates $x_i$ with $i \leq 0$. Therefore $\al(n,\pi_k(x)) = \al(n,\pi'_k(x))$ for all $n \geq 1$ and $x \in \Xtil$. Since
$$\xi_{n,k}(x) = \xi(n \cdot \pi_k(x)) = \al(n,\pi_k(x))(\xi(\pi_k(x)))$$
and similarly for $\xi'_{n,k}$, we get that $\|\xi_{n,k} - \xi'_{n,k}\|_2 = \|\xi \circ \pi_k - \xi \circ \pi'_k\|_2$ for all $n \geq 1$. Using \eqref{eq.good-limit}, it follows that $\xi \circ \pi_k = \xi \circ \pi'_k$ a.e.

When $k \leq 0$, we use that for all $n \leq -1$, the map $x \mapsto \al(n,x)$ only depends on the coordinates $x_i$ with $i \geq 1$. The same argument gives that $\xi \circ \pi_k = \xi \circ \pi'_k$ a.e.
\end{proof}

Applying Lemma \ref{lem.random-ergodic} to the Koopman representation of a pmp action, the following result immediately follows. It says that a $1$-cocycle as in \eqref{eq.my-1-cocycle} is \emph{weakly mixing} whenever the essential ranges of the maps $\vphi_i$ generate a dense subgroup. More precisely, we have the following.

\begin{proposition}\label{prop.skew-ergodicity}
Let $G$ be an lcsc group and $(X_0,\mu_0)$ a standard probability space. Let $\vphi_i : X_0 \to G$ be Borel maps, for $i = 1,\ldots,k$, and consider the $1$-cocycle $\al : \Gamma \times X \to G$ given by \eqref{eq.my-1-cocycle} for the Bernoulli action of the free group $\Gamma = \F_k$.

If $G \actson (Z,\zeta)$ is any pmp action and $\Gamma \actson X \times Z : g \cdot (x,z) = (g \cdot x, \al(g,x) \cdot z)$ is the skew product, then $L^\infty(X \times Z)^\Gamma = 1 \ot L^\infty(Z)^{G_0}$, where $G_0$ is the closed subgroup of $G$ generated by the essential ranges of the maps $\vphi_i$, $i = 1,\ldots,k$.
\end{proposition}

Lemma \ref{lem.random-ergodic} is usually referred to as Kakutani's random ergodic theorem. The following is thus the corresponding ``random strong ergodicity theorem''.

\begin{lemma}\label{lem.random-strongly-ergodic}
Let $\cH$ be a separable Hilbert space, $(X_0,\mu_0)$ a standard probability space and $\vphi_i : X_0 \to \cU(\cH)$ Borel maps, for $i \in \{1,\ldots,k\}$ and $k \geq 2$. Define $\Gamma = \F_k$ and the $1$-cocycle $\al : \Gamma \times X \to \cU(\cH)$ for the Bernoulli action $\Gamma \actson (X,\mu) = (X_0,\mu_0)^\Gamma$ given by \eqref{eq.my-1-cocycle}. Consider the skew product unitary representation
$$\pi : \Gamma \to \cU(L^2(X,\mu) \ot \cH) : (\pi(g)^* \xi)(x) = \al(g,x)^*(\xi(g \cdot x)) \quad\text{for all $g \in \Gamma$, $\xi \in L^2(X,\mu) \ot \cH$.}$$
If $\xi_n \in L^2(X,\mu) \ot \cH$ is a sequence of approximately $\pi$-invariant unit vectors and $\eta_n \in \cH$ is defined by $\eta_n = \int_X \xi_n(x) \, d\mu(x)$, then $\|\xi_n - 1 \ot \eta_n\|_2 \to 0$.
\end{lemma}
\begin{proof}
Restricting the action to $\F_2 \subset \F_k$, it suffices to prove the lemma for $k=2$. We denote by $a=a_1$ and $b=a_2$ the two free generators of $\F_2$. We again identify $L^2(X,\mu) \ot \cH = L^2(X,\cH)$.

Let $\cC$ be the set of all finite nonempty subsets of $\F_2$. We partition $\cC = \cC_1 \sqcup \cC_2$ where $\cC_1$ consists of $\{e\}$ and the finite nonempty subsets $\cF \subset \F_2$ with the following two properties: $\cF \neq \{e\}$ and one of the largest words in $\cF$ starts with a negative power of $a$ or $b$. Then $\cC_2$ consists of the finite nonempty subsets $\cF \subset \F_2$ with the following two properties: $\cF \neq \{e\}$ and all the largest words in $\cF$ start with a positive power of $a$ or $b$.

For every nonempty finite subset $\cF \subset \F_2$, we define the set of vectors $\cS(\cF) \subset L^2(X,\cH)$ consisting of all $\xi \in L^2(X,\cH)$ of the form
\begin{equation}\label{eq.my-xi}
\begin{split}
\xi(x) = \Bigl(\prod_{h \in \cF} \xi_h(x_h)\Bigr) \, \eta \quad &\text{where $\eta \in \cH$, $\|\eta\|=1$, $\xi_h \in L^2(X_0,\mu_0)$, $\|\xi_h\|_2 = 1$ and}\\ &\text{$\int_{X_0} \xi_h \, d\mu_0 = 0$ for all $h \in \cF$.}
\end{split}
\end{equation}
For $i \in \{1,2\}$, define $\cK_i$ as the closed linear span of all $\cS(\cF)$ with $\cF \in \cC_i$. Then, $L^2(X,\cH) = \cK_1 \oplus \cK_2 \oplus (1 \ot \cH)$.

Define the subgroup $\Gamma_1 \subset \F_2$ as $\Gamma_1 = \langle ab^{-1}, a^2 b^{-2} \rangle$. Also define $\Gamma_2 = \langle a^{-1} b , a^{-2} b^2 \rangle$. Note that $\Gamma_i \cong \F_2$. Also note that every $g \in \Gamma_1 \setminus \{e\}$ is a word that starts with a positive power of $a$ or $b$ and that ends with a negative power of $a$ or $b$. Similarly, every $g \in \Gamma_2 \setminus \{e\}$ is a word that starts with a negative power of $a$ or $b$ and that ends with a positive power of $a$ or $b$.

For every $i \in \{1,2\}$, denote by $\cL_i$ the closed linear span of $\pi(\Gamma_i) \cK_i$. We prove below that the unitary representation $\pi(\Gamma_i)$ of $\Gamma_i \cong \F_2$ on $\cL_i$ is a multiple of the regular representation. Since $\Gamma_i$ is nonamenable, it then follows that $\|P_{\cK_i}(\xi_n)\|_2 \to 0$ for $i \in \{1,2\}$. Since $L^2(X,\cH) = \cK_1 \oplus \cK_2 \oplus (1 \ot \cH)$, this implies that $\|\xi_n - 1 \ot \eta_n\|_2 \to 0$.

By symmetry, we only prove that the representation $\pi(\Gamma_1)$ on $\cL_1$ is a multiple of the regular representation. To prove this, it suffices to fix $\cF \in \cC_1$, $\xi \in \cS(\cF)$, $g \in \Gamma_1 \setminus \{e\}$ and prove that $\langle \pi(g) \xi,\xi\rangle = 0$. Write $\xi$ as in \eqref{eq.my-xi}. Then
$$(\pi(g)\xi)(x) = \Bigl(\prod_{h \in g \cF} \xi_{g^{-1}h}(x_h)\Bigr) \, \al(g^{-1},x)^*(\eta) \; .$$
If $\cF = \{e\}$, write $h_0 = e$. If $\cF \neq \{e\}$, choose $h_0 \in \cF$ such that the word $h_0$ starts with a negative power of $a$ or $b$ and $|h| \leq |h_0|$ for all $h \in \cF$.
Since $g^{-1}$ starts with a positive power of $a$ or $b$, the function $x \mapsto \al(g^{-1},x)$ only depends on coordinates $x_k$ with $|k| < |g|$. Since $g$ ends with a negative power of $a$ or $b$, we have that $g h_0$ is a reduced word, with $|gh_0| = |g| + |h_0|$. So $|gh_0| > |h|$ for all $h \in \cF$, implying that $gh_0 \not\in \cF$. Also $|gh_0| \geq |g|$, implying that $x \mapsto \al(g^{-1},x)$ does not depend on the coordinate $x_{g h_0}$. We conclude that the integral of the function $x \mapsto \langle (\pi(g)\xi)(x) ,\xi(x) \rangle$ over the variable $x_{gh_0}$ equals zero. In particular, $\langle \pi(g) \xi,\xi\rangle = 0$.
\end{proof}

In the same way as Proposition \ref{prop.skew-ergodicity} is an immediate consequence of Lemma \ref{lem.random-ergodic}, the following result is an immediate consequence of Lemma \ref{lem.random-strongly-ergodic}.

\begin{proposition}\label{prop.skew-strong-ergodicity}
Let $G$ be an lcsc group and $(X_0,\mu_0)$ a standard probability space. Let $\vphi_i : X_0 \to G$ be Borel maps, for $i = 1,\ldots,k$, and consider the $1$-cocycle $\al : \Gamma \times X \to G$ given by \eqref{eq.my-1-cocycle} for the Bernoulli action of the free group $\Gamma = \F_k$. Denote $\nu_i = (\vphi_i)_*(\mu_0)$ and define $G_0$ as the subgroup of $G$ generated by the supports of the measures $\nu_i$, $i \in \{1,\ldots,k\}$. Assume that $G_0$ is dense in $G$ and that at least one of the $\nu_i$ is not singular w.r.t.\ the Haar measure of $G$.

Then $G_0 = G$ and for any strongly ergodic pmp action $G \actson (Z,\zeta)$, the skew product action $\Gamma \actson X \times Z : g \cdot (x,z) = (g \cdot x, \al(g,x) \cdot z)$ remains strongly ergodic.
\end{proposition}
\begin{proof}
Let $\lambda$ be a left Haar measure on $G$. Denote by $S_i \subset G$ the support of $\nu_i$. Assume that $\nu_{i_0}$ is not singular w.r.t.\ $\lambda$. Then $\lambda(S_{i_0}) > 0$. A fortiori, $\lambda(G_0) > 0$. Since $G_0$ is a dense subgroup of $G$, it follows that $G_0 = G$.

Let $G \actson (Z,\zeta)$ be a strongly ergodic pmp action. Let $F_n \in L^\infty(X \times Z)$ be a bounded sequence such that, w.r.t.\ the skew product action $\Gamma \actson X \times Z$, we have that $F_n(g \cdot (x,z)) - F_n(x,z) \to 0$ for all $g \in \Gamma$ and a.e.\ $(x,z) \in X \times Z$. Write $H_n = (\mu \ot \id)(F_n)$. By Lemma \ref{lem.random-strongly-ergodic} and after passage to a subsequence, we have that $F_n(x,z) - H_n(z) \to 0$ for a.e.\ $(x,z) \in X \times Z$.

Expressing the approximate invariance of $F_n$ under the generators $a_i \in \Gamma$ and using the Fubini theorem, we find Borel sets $T_i \subset S_i$ such that $\nu_i(S_i \setminus T_i) = 0$ and for all $h \in T_i$,
\begin{equation}\label{eq.we-have-this}
H_n(h \cdot z) - H_n(z) \to 0 \quad\text{for a.e.\ $z \in Z$.}
\end{equation}
Since $S_i$ is the support of $\nu_i$, we get that $T_i$ is dense in $S_i$. So, the subgroup $G_1 \subset G$ generated by $T_1 \cup \cdots \cup T_k$ is dense in $G$. Since $\nu_{i_0}$ is not singular w.r.t.\ $\lambda$, also $\lambda(T_{i_0}) > 0$. It follows that $\lambda(G_1) > 0$ and thus $G_1 = G$. It then follows that \eqref{eq.we-have-this} holds for all $h \in G$. By Lemma \ref{lem.criteria-strongly-ergodic} and after a further passage to a subsequence, we find a bounded sequence $t_n \in \C$ such that $H_n(z) - t_n \to 0$ for a.e.\ $z \in Z$. It follows that $F_n(x,z) - t_n \to 0$ for a.e.\ $(x,z) \in X \times Z$. By Lemma \ref{lem.criteria-strongly-ergodic}, the skew product action $\Gamma \actson X \times Z$ is strongly ergodic.
\end{proof}

\subsection{Construction of measure equivalence embeddings}

Given a standard probability space $(X_0,\mu_0)$ and a standard Borel space $Z$, recall that Borel maps $\vphi_i : X_0 \to Z$, $i \in \{1,\ldots,k\}$, are called \emph{independent} if
$$\mu_0(\vphi_1^{-1}(U) \cap \cdots \cap \vphi_k^{-1}(U)) = \mu_0(\vphi_1^{-1}(U)) \cdots \mu_0(\vphi_k^{-1}(U))$$
for every Borel set $U \subset Z$.

Given an lcsc group $G$, we denote by $L^2(G)$ the Hilbert space of square integrable functions w.r.t.\ the left Haar measure and denote, for every probability measure $\nu$ on $G$, by $\lambda(\nu)$ the convolution operator on $L^2(G)$ given by
$$(\lambda(\nu) \xi)(g) = \int_G \xi(k^{-1}g) \, d\nu(k) \quad\text{for all $\xi \in L^2(G)$ and $g \in G$.}$$

\begin{proposition}\label{prop.ME-embedding}
Let $G$ be an lcsc group and $k \in \N$. Let $(X_0,\mu_0)$ be a standard probability space and $\vphi_i : X_0 \to G$ Borel maps for $i \in \{1,\ldots,k\}$. Denote $\Gamma = \F_k$ and consider the $1$-cocycle $\al : \Gamma \times X \to G$ for the Bernoulli action $\Gamma \actson (X,\mu) = (X_0,\mu_0)^\Gamma$ given by \eqref{eq.my-1-cocycle}. Consider the skew product action $\Gamma \actson X \times G : g \cdot (x,h) = (g \cdot x, \al(g,x) h)$.
\begin{enumlist}
\item ($k=1$) If the random walk on $G$ defined by the probability measure $\nu = (\vphi_1)_*(\mu_0)$ is transient, meaning that $\sum_{k=1}^\infty \nu^{\ast k}(K) < +\infty$ for every compact subset $K \subset G$, then the skew product action $\Gamma \actson X \times G$ admits a fundamental domain.
\item ($k \geq 2$) If the maps $\vphi_i$ are independent and $\|\lambda((\vphi_i)_*(\mu_0))\| < (2k-1)^{-1}$ for all $i \in \{1,\ldots,k\}$, then the skew product action $\Gamma \actson X \times G$ admits a fundamental domain.
\end{enumlist}
\end{proposition}
\begin{proof}
Let $\lambda$ be a left Haar measure on $G$. For every $g \in \Gamma$, we have that
$$(\mu \times \lambda)\bigl( (X \times K) \cap g^{-1} \cdot (X \times K)\bigr) = \int_X \langle \lambda(\al(g,x)) 1_K , 1_K \rangle \, d\mu(x) \; .$$
When $k = 1$ and $\nu = (\vphi_1)_*(\mu_0)$, we write $\Gamma = \Z$ and get for all $n \geq 1$ that
$$(\mu \times \lambda)\bigl( (X \times K) \cap (-n) \cdot (X \times K)\bigr) = \langle \lambda(\nu^{\ast n}) 1_K , 1_K \rangle \leq \lambda(K) \, \nu^{\ast n}(K K^{-1}) \; .$$
When $n \leq -1$, a similar estimate holds. By transience,
$$\sum_{n \in \Z} (\mu \times \lambda)\bigl( (X \times K) \cap (-n) \cdot (X \times K)\bigr) < + \infty$$
for every compact subset $K \subset G$. So, $X \times K$ is contained in the dissipative part of $\Z \actson X \times G$ for every compact $K \subset G$. So, $\Z \actson X \times G$ admits a fundamental domain.

When $k \geq 2$, we write $\nu_i = (\vphi_i)_*(\mu_0)$ and $\rho = \max \bigl\{ \|\lambda(\nu_i)\| \bigm| i \in \{1,\ldots,k\} \bigr\}$, so that $\rho < (2k-1)^{-1}$. The same computation, using the independence of the maps $\vphi_i$, gives
$$(\mu \times \lambda)\bigl( (X \times K) \cap g^{-1} \cdot (X \times K)\bigr) = \langle T_1 \cdots T_{|g|} 1_K, 1_K \rangle \; ,$$
where each of the operators $T_j$ is of the form $\lambda(\nu_i)$ or $\lambda(\nu_i)^*$ for some $i \in \{1,\ldots,k\}$. It follows that
$$(\mu \times \lambda)\bigl( (X \times K) \cap g^{-1} \cdot (X \times K)\bigr) \leq \lambda(K) \, \rho^{|g|}$$
for all $g \in \Gamma$. Since for every $n \geq 1$, the number of elements in $\{g \in \F_k \mid |g| = n\}$ equals $2 k (2k-1)^{n-1}$, we again conclude that
$$\sum_{g \in \Gamma} (\mu \times \lambda)\bigl( (X \times K) \cap g^{-1} \cdot (X \times K)\bigr) < + \infty \; ,$$
so that $\Gamma \actson X \times G$ admits a fundamental domain.
\end{proof}

\subsection{Proof of Theorems \ref{thm.main-ME-embedding}, \ref{thm.ME-embedding-nonamenable} and \ref{thm.ME-embedding-noncompact}}\label{sec.proof-A}

\begin{proof}[{Proof of Theorem \ref{thm.main-ME-embedding}}]
By Proposition \ref{prop.ME-embedding}, the given action $\Gamma \times G \actson \Om$ defines a measure equivalence embedding of $\Gamma$ into $G$. Write $Y = \Gamma \backslash \Omega$ with the natural action $G \actson Y$. Let $G \actson (Z,\zeta)$ be an ergodic pmp action. By Proposition \ref{prop.skew-ergodicity}, the skew product action $\Gamma \actson X \times Z$ is ergodic. This means that the diagonal action $G \actson Y \times Z$ is ergodic. So, $G \actson Y$ is weakly mixing. If $G \actson (Z,\zeta)$ is a strongly ergodic pmp action, it follows from Proposition \ref{prop.skew-strong-ergodicity} that the skew product action $\Gamma \actson X \times Z$ is strongly ergodic. Applying Proposition \ref{prop.stability-of-strong-ergodicity-under-correspondence} to the correspondence $\Gamma \times G \actson \Omega \times Z$, this precisely means that the diagonal action $G \actson Y \times Z$ is strongly ergodic. So, $G \actson Y$ is stably strongly ergodic.
\end{proof}

\begin{proof}[{Proof of Theorem \ref{thm.ME-embedding-nonamenable}}]
The implication $3\Rightarrow 2$ is trivial. As explained below \cite[Theorem B]{GM15} the implication $2\Rightarrow 1$ also holds true. So we only need to prove that $1 \Rightarrow 3$. Let $\nu_0$ be any probability measure on $G$ that is equivalent with a left Haar measure $\lambda$ on $G$. By \cite[Th\'{e}or\`{e}me 4]{BC73}, we have that $\|\lambda(\nu_0)\| < 1$. Defining $\nu = \nu_0^{\ast k}$ for sufficiently large $k$, we get that $\|\lambda(\nu)\| < 1/3$. We can now apply Theorem \ref{thm.main-ME-embedding} to conclude that 3 holds.
\end{proof}

The proof of Theorem \ref{thm.ME-embedding-noncompact} follows the same lines as the proof of Theorem \ref{thm.ME-embedding-nonamenable}, once we have proven the following lemma, which is probably well known, although we could not find a reference for this result. Recall that a probability measure $\nu$ on an lcsc group defines a transient random walk if and only if $\sum_{n=1}^\infty \nu^{\ast n}(K) < +\infty$ for all compact subsets $K \subset G$.

\begin{lemma}\label{lem.transient-random-walk}
Every noncompact lcsc group $G$ admits a symmetric probability measure that is equivalent with the Haar measure and defines a transient random walk on $G$.
\end{lemma}

\begin{proof}
Choose a left Haar measure $\lambda$ on $G$. Denote by $\Delta$ the modular function on $G$.

When $G$ is nonamenable, any probability measure $\nu$ on $G$ that is equivalent with the Haar measure satisfies $\|\lambda(\nu)\| < 1$ by \cite[Th\'{e}or\`{e}me 4]{BC73}. Any probability measure $\nu$ on $G$ with $\|\lambda(\nu)\| < 1$ defines a transient random walk. Indeed, when $K \subset G$ is a compact subset and $L$ is any compact neighborhood of $e$ in $G$, we get that $\lambda(g L \cap KL) = \lambda(L) > 0$ for all $g \in K$. Thus,
$$\lambda(L) \, \nu^{\ast n}(K) \leq \langle \lambda(\nu^{\ast n}) 1_L, 1_{KL} \rangle \leq \|\lambda(\nu)\|^n \, \|1_L\|_2 \, \|1_{KL}\|_2 \; ,$$
so that transience follows.

Next assume that $G$ is amenable. We follow very closely the proof of \cite[Theorem 4.3]{KV83}. Since for $n\geq 3$ the sequence $n^{-2/n}$ is increasing to $1$, we can define the sequence $(t_n)_{n\geq 1}$ of positive numbers by
\begin{align*}
t_n=\begin{cases}3^{-5/3}&\text{ if } n\in \{1,2,3\}\\
n^{-2/n}-(n-1)^{-2/(n-1)}&\text{ if } n\geq 4\end{cases} \; .
\end{align*}
By construction, $\sum_{k=1}^{\infty}t_k=1$ and $(\sum_{k=1}^{n}t_k)^{n}=n^{-2}$ for every $n\geq 3$. Set $\eps_n=1$ if $n=1,2$ and $\eps_n=n^{-2}$ if $n\geq 3$. Let $B_n$ be an increasing sequence of compact subsets of $G$ with interior $B_n^\circ$ such that $\bigcup_{n=1}^{\infty} B_n^\circ =G$. Because $G$ is not compact, for every $n\geq 1$ we can choose a finite subset $\mathcal{F}_n\subset G$ such that $|\mathcal{F}_n|>\eps_n^{-1}$ and such that $(g^{-1}\cdot B_n)_{g\in \mathcal{F}_n}$ is a family of disjoint sets.

Since $G$ is amenable, we can inductively choose symmetric probability measures $\nu_n$ on $G$ with compact support $A_n$ such that $B_n \subset A_n$, such that $\nu_n$ is equivalent with the restriction of $\lambda$ to $A_n$ and
\begin{align*}
\|g\nu_n-\nu_n\| \leq \eps_n \quad\text{for every}\;\; g\in \Bigl(\{e\}\cup\bigcup_{j=1}^{n}\mathcal{F}_j\Bigr)\;\Bigl(\{e\}\cup\bigcup_{k=1}^{n-1}A_k\Bigr)^{n-1} \; .
\end{align*}
Here $\|\cdot \|$ denotes the total variation norm. Set $\nu=\sum_{n=1}^{\infty} t_n\nu_n$. By construction $\nu$ is a symmetric probability measure on $G$ that is equivalent with $\lambda$. We claim that the random walk associated to $\nu$ is transient. Just as in the proof of \cite[Theorem 4.3]{KV83} one shows that for every $n\geq 4$ we have that
\begin{align}\label{eq:almost invariance convolution powers}
\|g\nu^{\ast n}-\nu^{\ast n}\|\leq 4\eps_n \quad\text{for every $g\in \mathcal{F}_n$.}
\end{align}
We also have that $\nu^{\ast n}(B_n)\leq 5\eps_n$ for every $n\geq 4$. Indeed, assume instead that $\nu^{\ast n}(B_n)> 5\eps_n$ for some $n\geq 4$. Then it follows from \eqref{eq:almost invariance convolution powers} that $\nu^{\ast n}(g^{-1}B_n)=(g\nu^{\ast n})(B_n)>\eps_n$ for every $g\in \mathcal{F}_n$. But the sets $g^{-1} B_n$ are pairwise disjoint and we get that
\begin{align*}
1\geq \sum_{g\in \mathcal{F}_n}\nu^{\ast n}(g^{-1} B_n)>|\mathcal{F}_n| \, \eps_n>1 \; ,
\end{align*}
which is a contradiction. Finally, if $K\subset G$ is an arbitrary compact set, then for $n$ large enough we have that $K\subset B_n$. Since $\sum_{n=1}^{\infty}\nu^{\ast n}(B_n)\leq 5\sum_{n=1}^{\infty}\eps_n<+\infty$, it follows that $\sum_{n=1}^{\infty}\nu^{\ast n}(K)<+\infty$.
\end{proof}

\begin{proof}[{Proof of Theorem \ref{thm.ME-embedding-noncompact}}]
The implication $3\Rightarrow 2$ is trivial.

$2\Rightarrow 1$. Suppose that $G$ is compact and that $G$ admits $\Z$ as an ME subgroup. Then we find a pmp action $\Z\actson (X,\mu)$ and a cocycle $\alpha : \Z\times X\rightarrow G$ such that the skew product action $\Z\actson X \times G : n\cdot(x,g)=(n\cdot x,\al(n,x)g)$ admits a fundamental domain. Since this action is also pmp and $\Z$ is infinite, this is absurd.

$1 \Rightarrow 3$. By Lemma \ref{lem.transient-random-walk}, we can take a transient probability measure $\nu$ on $G$ that is equivalent with the Haar measure. Define $(X,\mu) = (G,\nu)^\Z$ with the Bernoulli action $\Z \actson X$ and $\Om = X \times G$. Then define
$$\Z \times G \actson \Om : (n,g) \cdot (x,k) = (n \cdot x, \al(n,x) k g^{-1}) \; ,$$
where $\al : \Z \times X \to G$ is the unique $1$-cocycle satisfying $\al(1,x) = x_0$. By Proposition \ref{prop.ME-embedding}, the action $\Z \times G \actson \Om$ defines an ME embedding of $\Z$ into $G$. Weak mixing follows from Proposition \ref{prop.skew-ergodicity}.
\end{proof}

\section{Application: every nonamenable lcsc group admits strongly ergodic actions of any type}\label{sec.proof-C}

The first two statements of Theorem \ref{thm.main-actions-nonamenable} will be deduced as an immediate consequence of the following more general result. The last statement will similarly follow from Theorem \ref{thm.very-general-actions-associated-flow} below.

\begin{theorem}\label{thm.very-general-actions}
Let $G$ be any nonamenable lcsc group, $k \geq 1$, $\Gamma = \F_{k+2}$ and $\pi : \Gamma \to \F_k$ the natural surjective homomorphism using the first $k$ generators. There exists a $1$-cocycle $\al : \Gamma \times X \to G$ for the Bernoulli action $\Gamma \actson (X,\mu) = (X_0,\mu_0)^\Gamma$ of the form \eqref{eq.my-1-cocycle} such that the following holds.
\begin{enumlist}
\item For any nonsingular ergodic action $\F_k \actson (U,\phi)$, the action
$$\Gamma \actson U \times X \times G : g \cdot (u,x,h) = (\pi(g) \cdot u, g \cdot x, \al(g,x)h)$$
admits a fundamental domain.
\item The resulting action $G \actson Y = \Gamma \backslash (U \times X \times G)$ by right translation in the last variable is essentially free, weakly mixing and nonamenable.
\item For every ergodic pmp action $G \actson (Z,\zeta)$, the flow of weights of $L^\infty(Y \times Z) \rtimes G$ is isomorphic with the associated flow of $\F_k \actson U$.
\item If $\F_k \actson U$ is strongly ergodic, then $G \actson Y$ is stably strongly ergodic: for every strongly ergodic pmp action $G \actson (Z,\zeta)$, the diagonal action $G \actson Y \times Z$ is strongly ergodic.
\end{enumlist}
\end{theorem}
\begin{proof}
As in the proof of Theorem \ref{thm.ME-embedding-nonamenable}, choose a probability measure $\nu$ on $G$ with full support such that $\nu$ is absolutely continuous w.r.t.\ the Haar measure of $G$ and $\|\lambda(\nu)\| \leq (2k+3)^{-1}$. Define $(X_0,\mu_0) = (G,\nu)^{k+2}$ and define the maps $\vphi_i : X_0 \to G : \vphi_i(g) = g_i$ for all $i \in \{1,\ldots,k+2\}$. Then consider the Bernoulli action $\Gamma \actson (X,\mu) = (X_0,\mu_0)^\Gamma$ and the $1$-cocycle $\al : \Gamma \times X \to G$ given by \eqref{eq.my-1-cocycle}.

We view $\Gamma = \Lambda_1 \ast \Lambda_2$ where $\Lambda_1 \cong \F_k$ is generated by the first $k$ generators of $\F_{k+2}$ and $\Lambda_2$ by the last two generators. Note that $\pi(g) = g$ for all $g \in \Lambda_1$ and $\pi(g) = e$ for all $g \in \Lambda_2$.

1.\ By Proposition \ref{prop.ME-embedding}, the skew product action $\Gamma \actson X \times G$ admits a fundamental domain. A fortiori, the given action $\Gamma \actson U \times X \times G$ admits a fundamental domain.

2.\ Let $G \actson (Z,\zeta)$ be any ergodic pmp action. We have to prove that
\begin{equation}\label{eq.my-good-action}
\Gamma \actson U \times X \times Z : g \cdot (u,x,z) = (\pi(g) \cdot u, g \cdot x, \al(g,x) \cdot z)
\end{equation}
is ergodic and nonamenable. The restriction of this action to $\Lambda_2$ is pmp and $\Lambda_2$ is a nonamenable group. So the action $\Lambda_2 \actson U \times X \times Z$ is nonamenable. A fortiori, $\Gamma \actson U \times X \times Z$ is nonamenable. By Proposition \ref{prop.skew-ergodicity}, $L^\infty(U \times X \times Z)^{\Lambda_2} = L^\infty(U) \ot 1 \ot 1$. Then ergodicity follows because $\Lambda_1 \actson U$ is ergodic.

3.\ The flow of weights of $L^\infty(Y \times Z) \rtimes G$ is given by the natural action of $\R$ on the ergodic decomposition of the Maharam extension of $\Gamma \actson U \times X \times Z$. Since the skew product action on $X \times Z$ is measure preserving, this Maharam extension is given by $\Gamma \actson \Util \times X \times Z$, where $\Lambda_1 \actson \Util$ is the Maharam extension of $\Lambda_1 \actson U$. But then the same argument as in 2 identifies $L^\infty(\Util \times X \times Z)^\Gamma$ with $L^\infty(\Util)^{\Lambda_1}$.

4.\ As in the proof of Theorem \ref{thm.ME-embedding-nonamenable}, assuming that $\Lambda_1 \actson U$ and $G \actson Z$ are strongly ergodic, we have to prove that the action in \eqref{eq.my-good-action} is strongly ergodic. Take a bounded sequence $F_n \in L^\infty(U \times X \times Z)$ such that $F_n(g \cdot (u,x,z)) - F_n(u,x,z) \to 0$ for all $g \in \Gamma$ and a.e.\ $(u,x,z) \in U \times X \times Z$. By Proposition \ref{prop.skew-strong-ergodicity}, the skew product action $\Lambda_2 \actson X \times Z$ is strongly ergodic. By Lemma \ref{lem.criteria-strongly-ergodic} and after passage to a subsequence, we find a bounded sequence $H_n \in L^\infty(U)$ such that $F_n(u,x,z) - H_n(u) \to 0$ for a.e.\ $(u,x,z) \in U \times X \times Z$. Expressing approximate invariance of $F_n$ under $\Lambda_1$, using that $\Lambda_1 \actson U$ is strongly ergodic and applying Lemma \ref{lem.criteria-strongly-ergodic}, we can again pass to a subsequence such that $H_n(u) - t_n \to 0$ for a.e.\ $u \in U$, where $(t_n)_n$ is a bounded sequence of complex numbers. So, $F_n(u,x,z) - t_n \to 0$ for a.e.\ $(u,x,z) \in U \times X \times Z$, proving strong ergodicity of the action \eqref{eq.my-good-action}.
\end{proof}

Recall from Section \ref{sec.flows} that when $G$ is a nonunimodular lcsc group, the flow of weights of a crossed product $L^\infty(Y) \rtimes G$ is typically different from the associated flow of $G \actson Y$. In particular, recall from Section \ref{sec.flows} that the associated flow governs the existence of an equivalent $G$-invariant measure and that for nonunimodular groups $G$, a pmp action may give rise to a type III crossed product.

The following variant of Theorem \ref{thm.very-general-actions} shows that also this associated flow can be arbitrary.

\begin{theorem}\label{thm.very-general-actions-associated-flow}
Let $G$ be any nonamenable lcsc group, $k \geq 1$, $\Gamma = \F_{k+3}$ and $\pi : \Gamma \to \F_k$ the natural surjective homomorphism using the first $k$ generators. There exists a $1$-cocycle $\al : \Gamma \times X \to G$ for a \emph{nonsingular} Bernoulli action $\Gamma \actson (X,\mu)$, still of the form \eqref{eq.my-1-cocycle}, such that the following holds.
\begin{enumlist}
\item For any nonsingular ergodic action $\F_k \actson (U,\phi)$, the action $\Gamma \actson U \times X \times G : g \cdot (u,x,h) = (\pi(g) \cdot u, g \cdot x, \al(g,x)h)$ admits a fundamental domain.
\item The resulting action $G \actson Y = \Gamma \backslash (U \times X \times G)$ by right translation in the last variable is essentially free, weakly mixing and nonamenable.
\item For every ergodic pmp action $G \actson (Z,\zeta)$, the associated flow of $G \actson Y \times Z$ is isomorphic with the associated flow of $\F_k \actson U$.
\item If $\F_k \actson U$ is strongly ergodic, then $G \actson Y$ is stably strongly ergodic: for every strongly ergodic pmp action $G \actson (Z,\zeta)$, the diagonal action $G \actson Y \times Z$ is strongly ergodic.
\end{enumlist}
\end{theorem}
\begin{proof}
Denote by $\Delta$ the modular function of $G$ and let $\lambda$ be a left Haar measure on $G$.
Write $\delta = (2k+5)^{-1}$. We first prove there exist probability measures $\nu_1$ and $\nu_2$ on $G$ such that $\nu_1 \sim \nu_2 \sim \lambda$, $d\nu_1 / d\nu_2 = \Delta$ and $\|\lambda(\nu_i)\| < \delta$ for all $i \in \{1,2\}$.

Choose any strictly positive Borel function $H : G \to (0,+\infty)$ such that $H(g^{-1}) = H(g)$ for all $g \in G$ and $\int_G H \, d\lambda = 1$. Since $H$ is symmetric, also $\int_G H \Delta^{-1} \, d\lambda = 1$. We can then define the probability measures $\eta_1$ and $\eta_2$ on $G$ such that $\eta_1 \sim \eta_2 \sim \lambda$, $d\eta_1 / d \lambda = H$ and $d\eta_2 / d\lambda = H \Delta^{-1}$. By construction $d\eta_1 / d\eta_2 = \Delta$. Since $\eta_1$ and $\eta_2$ have full support, it follows from \cite[Th\'{e}or\`{e}me 4]{BC73} that $\|\lambda(\eta_i)\| < 1$ for all $i \in \{1,2\}$. Taking $k$ large enough, the convolution powers $\nu_i = \eta_i^{\ast k}$ are still equivalent to the Haar measure and satisfy $\|\lambda(\nu_i)\| < \delta$ for all $i \in \{1,2\}$. Since $\Delta$ is multiplicative, we still have that $d \nu_1 / d \nu_2 = \Delta$.

We denote the first $k+2$ free generators of $\F_{k+3}$ as $a_i$ with $i \in \{1,\ldots,k+2\}$. We denote the last ``extra'' generator as $a$. We denote by $\Lambda_1$ the subgroup generated by $a_1,\ldots,a_k$ and by $\Lambda_2$ the subgroup generated by $a_{k+1}, a_{k+2}$, so that $\Gamma = \Lambda_1 \ast \Lambda_2 \ast a^\Z$.

Denote by $G_0$ the kernel of $\Delta$. Since $G$ is nonamenable, $G_0$ is still nonamenable. As in the proof of Theorem \ref{thm.ME-embedding-nonamenable}, we can choose a probability measure $\nu_0$ on $G_0$ such that $\nu_0$ is equivalent with the Haar measure of $G_0$ and $\|\lambda(\nu_0)\|_{B(L^2(G_0))} < \delta$. Viewing $\nu_0$ as a probability measure on $G$, we also have that
$$\|\lambda(\nu_0)\|_{B(L^2(G))} = \|\lambda(\nu_0)\|_{B(L^2(G_0))} < \delta \; .$$
Define $X_0 = G^{k+3}$ and consider on $X_0$ the equivalent probability measures $\mu_1 = \nu_0^{k+2} \times \nu_1$ and $\mu_2 = \nu_0^{k+2} \times \nu_2$. Then define for every $g \in \Gamma$ the probability measure $\mu_g$ on $X_0$ by writing $\mu_g = \mu_1$ if the word $g$ ends with a positive power of the extra generator $a$ and $\mu_g = \mu_2$ otherwise. Consider the nonsingular Bernoulli action
$$\Gamma \actson (X,\mu) = \prod_{h \in \Gamma} (X_0,\mu_h) : (g \cdot x)_h = x_{g^{-1}h} \; .$$
For every $i \in \{1,\ldots,k+3\}$, define $\vphi_i : X_0 \to G : \vphi_i(x) = x_i$ for all $x \in X_0 = G^{k+3}$. Denote by $\al : \Gamma \times X \to G$ the unique $1$-cocycle satisfying $\al(a_i,x) = \vphi_i(x_e)$ for all $i \in \{1,\ldots,k+2\}$ and $\al(a,x) = \vphi_{k+3}(x_e)$.

Denote by $\om_X : \Gamma \times X \to \R$ the logarithm of the Radon-Nikodym $1$-cocycle. The first key observation is that, by construction, $\om_X(g,x) = \log \Delta(\al(g,x))$ for all $g \in \Gamma$ and a.e.\ $x \in X$. The second key observation is that the restriction of the action $\Gamma \actson X$ and the cocycle $\al$ to $\Lambda_2$ is a cocycle of the form \eqref{eq.my-1-cocycle} for the pmp Bernoulli action $\Lambda_2 \actson X$, to which our earlier results can be applied.

From now on, we can almost verbatim repeat the proof of Theorem \ref{thm.very-general-actions}. For clarity, we repeat the steps of the proof.

1.\ The proof of Proposition \ref{prop.ME-embedding} remains valid for nonsingular Bernoulli actions. So the skew product action $\Gamma \actson X \times G$ admits a fundamental domain. A fortiori, the given action $\Gamma \actson U \times X \times G$ admits a fundamental domain.

2.\ Let $G \actson (Z,\zeta)$ be any ergodic pmp action. We have to prove that
\begin{equation}\label{eq.my-good-action-bis}
\Gamma \actson U \times X \times Z : g \cdot (u,x,z) = (\pi(g) \cdot u, g \cdot x, \al(g,x) \cdot z)
\end{equation}
is ergodic and nonamenable. The restriction of this action to $\Lambda_2$ is pmp, so that nonamenability follows. By Proposition \ref{prop.skew-ergodicity}, $L^\infty(U \times X \times Z)^{\Lambda_2} = L^\infty(U) \ovt 1 \ovt L^\infty(Z)^{G_0}$. Expressing invariance under $a$ then gives that $L^\infty(U \times X \times Z)^{\Lambda_2 \ast a^\Z} = L^\infty(U) \ot 1 \ot 1$. Because the action $\F_k \actson U$ is ergodic, expressing invariance under $\Lambda_1$, we conclude that the action in \eqref{eq.my-good-action-bis} is ergodic.

3.\ Let $G \actson (Z,\zeta)$ be an ergodic pmp action. As explained in Section \ref{sec.flows}, the associated flow of $G \actson Y \times Z$ is isomorphic with the action of $\R$ on the ergodic decomposition of
\begin{equation}\label{eq.again-skew}
\Gamma \actson U \times \R \times X \times Z : g \cdot (u,s,x,z) = (\pi(g) \cdot u , \Om(g,(u,x,z)) + s , g \cdot x, \al(g,x) \cdot z) \; ,
\end{equation}
where $\Om(g,(u,x,z)) = \om(g,(u,x,z)) - \log \Delta(\al(g,x))$ and $\om$ is the logarithm of the Radon-Nikodym $1$-cocycle. Denoting by $\om_U$ and $\om_X$ the logarithm of the Radon-Nikodym $1$-cocycle of $\F_k \actson U$, resp.\ $\Gamma \actson X$, we have $\om(g,(u,x,z)) = \om_U(\pi(g),u) + \om_X(g,x)$. Since $\om_X(g,x) = \log \Delta(\al(g,x))$, we conclude that the action in \eqref{eq.again-skew} is given by
$$g \cdot (u,s,x,z) = (\pi(g) \cdot u , \om_U(\pi(g),u) + s , g \cdot x, \al(g,x) \cdot z) \; .$$
It then follows from Proposition \ref{prop.skew-ergodicity} that $L^\infty(U \times \R \times X \times Z)^{\Lambda_2} = L^\infty(U \times \R) \ovt 1 \ovt L^\infty(Z)^{G_0}$. Expressing the invariance under $a$, we get that
$$L^\infty(U \times \R \times X \times Z)^{\Lambda_2 \ast a^\Z} = L^\infty(U \times \R) \ot 1 \ot 1 \; ,$$
so that finally $L^\infty(U \times \R \times X \times Z)^\Gamma = L^\infty(U \times \R)^{\F_k} \ot 1 \ot 1$. This precisely means that the associated flow of $G \actson Y \times Z$ is isomorphic with the associated flow of $\F_k \actson U$.

4.\ Let $G \actson (Z,\zeta)$ be a pmp action. Assume that $\F_k \actson U$ and $G \actson Z$ are strongly ergodic. We have to prove that $G \actson Y \times Z$ is strongly ergodic. By Proposition \ref{prop.stability-of-strong-ergodicity-under-correspondence}, it suffices to prove that the action in \eqref{eq.my-good-action-bis} is strongly ergodic. We again apply the criteria of Lemma \ref{lem.criteria-strongly-ergodic}. Let $F_n \in L^\infty(U \times X \times Z)$ be a bounded sequence such that $F_n(g \cdot (u,x,z)) - F_n(u,x,z) \to 0$ for all $g \in \Gamma$ and a.e.\ $(u,x,z) \in U \times X \times Z$.

Applying Proposition \ref{prop.skew-strong-ergodicity} to the action of $\Lambda_2$, we can pass to a subsequence and take a bounded sequence $H_n \in L^\infty(U \times Z)$ such that $F_n(u,x,z) - H_n(u,z) \to 0$ for a.e.\ $(u,x,z) \in U \times X \times Z$. Expressing the approximate invariance of $F_n$ under the generator $a$, it follows that $H_n(u,g \cdot z) - H_n(u,z) \to 0$ for a.e.\ $(g,u,z) \in G \times U \times Z$. By Lemma \ref{lem.criteria-strongly-ergodic} and a further passage to a subsequence, we find a bounded sequence $K_n \in L^\infty(U)$ such that $H_n(u,z) - K_n(u) \to 0$ for a.e.\ $(u,z) \in U \times Z$. Using strong ergodicity of $\F_k \actson U$, we can now conclude the proof.
\end{proof}

\begin{proof}[{Proof of Theorem \ref{thm.main-actions-nonamenable}}]
By e.g.\ \cite[Proposition 7.1]{VW17}, the free group $\F_3$ admits strongly ergodic actions of type III$_\lambda$ for any $\lambda \in (0,1]$. So the second and third statement of Theorem \ref{thm.main-actions-nonamenable} follow from respectively Theorem \ref{thm.very-general-actions} and Theorem \ref{thm.very-general-actions-associated-flow}. The first statement follows similarly from Theorem \ref{thm.very-general-actions} because $\F_k$ admits nonsingular actions with any prescribed associated flow, for instance by starting from a type III$_1$ action and then applying the functorial construction of \cite[Proposition 3.4]{VV22}.
\end{proof}

Using Theorem \ref{thm.ME-embedding-noncompact}, we also get the following variant of Theorem \ref{thm.main-actions-nonamenable} for arbitrary noncompact lcsc groups, in which the first point is of course well known.

\begin{theorem}
Let $G$ be a noncompact lcsc group and let $\R \actson (Z,\zeta)$ be an arbitrary ergodic flow. Then $G$ admits nonsingular, essentially free, weakly mixing actions $G \actson (X,\mu)$ of the following kinds:
\begin{enumlist}
\item admitting an equivalent $G$-invariant probability measure $\mu' \sim \mu$,
\item admitting an equivalent $G$-invariant infinite $\sigma$-finite measure $\mu' \sim \mu$,
\item with $L^\infty(X \times Y) \rtimes G$ having flow of weights $\R \actson (Z,\zeta)$ for all ergodic pmp actions $G \actson (Y,\eta)$ and the diagonal action $G \actson X \times Y$,
\item with $G \actson X \times Y$ having associated flow $\R \actson (Z,\zeta)$ for all ergodic pmp actions $G \actson (Y,\eta)$ and the diagonal action $G \actson X \times Y$.
\end{enumlist}
\end{theorem}
\begin{proof}
The Gaussian action $G \actson (X,\mu)$ associated with an infinite multiple of the regular representation of $G$ is essentially free, pmp and mixing; see e.g.\ \cite[Remark 1.1]{KPV13}.

When $G = \Z$, then for instance the nonsingular Bernoulli action of \cite[Corollary 6.2]{VW17} is one of the many essentially free, ergodic actions $\Z \actson (U_0,\phi_0)$ whose Maharam extension $\Z \actson (U_1,\phi_1) = (U_0 \times \R,\phi_0 \times \lambda)$ is weakly mixing and infinite measure preserving. Using the functorial construction of \cite[Proposition 3.4]{VV22}, we obtain an essentially free, weakly mixing action $\Z \actson (U_2,\phi_2)$ with the property that for every ergodic pmp action $\Z \actson (Y,\eta)$, the associated flow of the diagonal action $\Z \actson U_2 \times Y$ is given by $\R \actson (Z,\zeta)$.

Let $G$ be any noncompact lcsc group. Fix a weakly mixing ME embedding of $\Z$ into $G$ as given by Theorem \ref{thm.ME-embedding-noncompact}, encoded by $\Z \times G \actson \Omega$. Using the notation of the previous paragraph, it then follows that for $i \in \{1,2\}$, the natural action $G \actson \Z \backslash (U_i \times \Omega)$ is essentially free and weakly mixing. When $i=1$ and $G$ is unimodular, this action admits an equivalent infinite $G$-invariant measure. When $i=2$, the diagonal product with any ergodic pmp action $G \actson (Y,\eta)$ has a crossed product with flow of weights $\R \actson (Z,\zeta)$. When moreover $G$ is unimodular, $\R \actson (Z,\zeta)$ is also the associated flow.

So, it only remains to prove 2 and 4 when $G$ is nonunimodular. Denote by $\Delta : G \to (0,+\infty)$ the modular function. Define $G_0 = \Ker \Delta$ and $G_1 = G/G_0$. Then, $G_0$ and $G_1$ are both noncompact lcsc groups and $G_1$ is abelian. Denote by $\pi : G \to G_1$ the quotient homomorphism. Let $G \actson (X_0,\mu_0)$ be an essentially free, pmp, mixing action, as in the first paragraph of the proof. Since $G_1$ is unimodular, by the previous paragraph, we may choose a weakly mixing, nonsingular action $G_1 \actson (X_1,\mu_1)$ with property 2, resp.\ property 4. Consider the action
$$G \actson (X,\mu) = (X_0 \times X_1,\mu_0 \times \mu_1) : g \cdot (x_0,x_1) = (g \cdot x_0 , \pi(g) \cdot x_1) \; .$$
Since $G \actson (X_0,\mu_0)$ is mixing and $G_0$ is noncompact, also $G_0 \actson (X_0,\mu_0)$ is mixing. It then follows that $G \actson (X,\mu)$ is essentially free, weakly mixing and that the action satisfies property 2, resp.\ property 4.
\end{proof}

\section{\boldmath Measure equivalence embeddings for II$_1$ factors}\label{sec.ME-embeddings-factors}

Recall that a Hilbert bimodule $\bim{A}{\cK}{Q}$ is said to be \emph{coarse} if it is isomorphic with a closed $A$-$Q$-subbimodule of a multiple of $\bim{A \ot 1}{L^2(A \ovt Q)}{1 \ot Q}$. This is equivalent to saying that the $*$-homomorphism $A \otalg Q\op \to B(\cK)$ extends to a normal $*$-homomorphism $A \ovt Q\op \to B(\cK)$.

\begin{definition}\label{def.ME-II1}
Let $A$ and $B$ be finite von Neumann algebras. A \emph{measure equivalence embedding} of $A$ into $B$ is a Hilbert $A$-$B$-bimodule $\bim{A}{\cK}{B}$ together with a finite von Neumann algebra $Q \subset \End_{A\text{-}B}(\cK)$ such that
\begin{itemlist}
\item the resulting bimodules $\bim{A}{\cK}{Q\op}$ and $\bim{Q}{\cK}{B}$ are coarse,
\item $\cK$ is finitely generated as a Hilbert $Q$-$B$-bimodule,
\item the left $A$-action on $\cK$ is faithful.
\end{itemlist}

We say that $(\bim{A}{\cK}{B},Q)$ is a \emph{measure equivalence} of $A$ and $B$ if moreover
\begin{itemlist}
\item $\cK$ is finitely generated as a Hilbert $A$-$Q\op$-bimodule,
\item the right $B$-action on $\cK$ is faithful.
\end{itemlist}
\end{definition}

Note that by taking $Q = \C$, we get the following: if $A$ embeds into a finite amplification of $B$, then $A$ also admits a measure equivalence embedding into $B$; and similarly, if $A$ and $B$ are virtually isomorphic, meaning that $A$ embeds with finite index into an amplification of $B$, they are measure equivalent. The converse does not hold, as we will see in Proposition \ref{prop.counterexample}.

Although this is not really essential for proving the results in this section, for simplicity, we assume in the rest of this section that all von Neumann algebras have a separable predual and that all Hilbert spaces are separable.

\subsection{Basic results on measure equivalence embeddings}

\begin{lemma}\label{lem.transitive}
Let $A$, $B$ and $C$ be finite von Neumann algebras.
\begin{enumlist}
\item If $(\bim{A}{\cH}{B},Q_1)$, resp.\ $(\bim{B}{\cK}{C},Q_2)$, are measure equivalence embeddings of $A$ into $B$, resp.\ $B$ into $C$, then $Q := (Q_1 \ot_B Q_2)\dpr \cong Q_1 \ovt Q_2$ and $(\bim{A}{(\cH \ot_B \cK)}{C},Q)$ is a measure equivalence embedding of $A$ into $C$.
\item If $(\bim{A}{\cH}{B},Q)$ is a measure equivalence, then also $(\bim{B}{\overline{\cH}}{A},Q\op)$ is a measure equivalence.
\end{enumlist}
In particular ``admitting a measure equivalence embedding'' is a transitive relation and ``being measure equivalent'' is an equivalence relation on the class of finite von Neumann algebras.
\end{lemma}
\begin{proof}
1.\ Write $P_i = Q_i\op$. Since we can view $\cH$ as a Hilbert $A$-$(P_1 \ovt B)$-bimodule and $\cK$ as a $B$-$(P_2 \ovt C)$-bimodule, the identification
$$\cH \ot_B \cK = \cH \ot_{(P_1 \ovt B)} (L^2(P_1) \ot \cK)$$
shows that the natural right action of $P_1 \otalg P_2 \otalg C$ is a normal representation of $P_1 \ovt P_2 \ovt C$. Similarly identifying
$$\cH \ot_B \cK = (L^2(Q_2) \ot \cH) \ot_{(Q_2 \ovt B)} \cK$$
shows that the natural left action of $Q_2 \otalg Q_1 \otalg A$ is a normal representation of $Q_2 \ovt Q_1 \ovt A$.

Take finite sets $\cF_1 \subset \cH$ and $\cF_2 \subset \cK$ such that $\cF_1 \cdot (P_1 \otalg B)$ and $\cF_2 \cdot (P_2 \otalg C)$ have a dense linear span in $\cH$, resp.\ $\cK$. We may assume that $\cF_1$ consists of left $B$-bounded vectors. It is then immediate that $\cF = \cF_1 \ot_B \cF_2$ generates $\cH \ot_B \cK$ as right Hilbert $P_1 \ovt P_2 \ovt C$-module.

2.\ This follows directly from the definition.
\end{proof}

For later usage, we also record the following obvious reformulation of the definition of a measure equivalence (embedding).

\begin{lemma}\label{lem.charac}
Let $A$ and $B$ be finite von Neumann algebras and $\bim{A}{\cH}{B}$ a Hilbert $A$-$B$-bimodule.

Then $(\bim{A}{\cH}{B},Q)$ is a measure equivalence embedding of $A$ into $B$ if and only if for $n \in \N$ large enough and $P = M_n(\C) \ot Q\op$, the bimodule $\bim{A}{\cH^{\oplus n}}{B}$ is isomorphic with $\bim{\al(A)}{p L^2(P \ovt B)}{1 \ot B}$ with the natural right $P$-action, where $p \in P \ovt B$ is a projection and $\al : A \to p(P \ovt B)p$ is a faithful unital normal $*$-homomorphism such that $\bim{\al(A)}{p L^2(P \ovt B)}{P \ot 1}$ is coarse.

We have that $(\bim{A}{\cH}{B},Q)$ is a measure equivalence of $A$ and $B$ if and only if for $n \in \N$ large enough and $P = M_n(\C) \ot Q\op$, the bimodule $\bim{A}{\cH^{\oplus n}}{B}$ is isomorphic with $\bim{\al(A)}{p L^2(P \ovt B)}{1 \ot B}$ with the natural right $P$-action and also isomorphic with $\bim{1 \ot A}{L^2(Q \ovt B) q}{\be(B)}$ with the natural left $Q$-action, where $p \in P \ovt B$ and $q \in Q \ovt A$ are projections and $\al : A \to p(P \ovt B)p$ and $\be : B \to q(Q \ovt A)q$ are faithful unital normal $*$-homomorphisms.
\end{lemma}
\begin{proof}
This is just a translation of the definition.
\end{proof}

The following construction provides a rich source of measure equivalence embeddings of group von Neumann algebras. We will use this in the proof of Proposition \ref{prop.ME-groups} to show that when a countable group $\Gamma$ admits a measure equivalence embedding into countable group $\Lambda$, then the same holds for the finite von Neumann algebras $L(\Gamma)$ and $L(\Lambda)$. We will also use this in the proof of Theorem \ref{thm.free-group-factor-ME} to show that $L(\F_2)$ admits a measure equivalence embedding into any nonamenable II$_1$ factor.

\begin{proposition}\label{prop.source-ME}
Let $\Gamma$ be a countable group and $\Gamma \actson (X,\mu)$ a pmp action. Write $P = L^\infty(X) \rtimes \Gamma$. Let $(B,\tau)$ be a tracial von Neumann algebra and $\om : \Gamma \times X \to \cU(B)$ a $1$-cocycle.

There is a unique normal $*$-homomorphism $\al : L(\Gamma) \to P \ovt B : \al(u_g) = (u_g \ot 1) \, (x \mapsto \om(g,x))$ for all $g \in \Gamma$, where we view $x \mapsto \om(g,x)$ as a unitary element in $L^\infty(X) \ovt B \subset P \ovt B$.

If for a.e.\ $x \in X$, we have that $\sum_{g \in \Gamma} |\tau(\om(g,x))|^2 < +\infty$, then $\cH = \bim{\al(L(\Gamma))}{L^2(P \ovt B)}{1 \ot B}$ together with the natural right $P$-action is a measure equivalence embedding of $L(\Gamma)$ into $B$.

If there moreover exists a finite subset $\cF \subset L^2(X,L^2(B))$ such that for a.e.\ $x \in X$, the set
\begin{equation}\label{eq.condition-ME}
\bigl\{ \om(g,g^{-1} \cdot x) \, \vphi(g^{-1} \cdot x) \bigm| g \in \Gamma, \vphi \in \cF \bigr\} \quad\text{is total in $L^2(B)$,}
\end{equation}
then we have a measure equivalence between $L(\Gamma)$ and $B$.
\end{proposition}
\begin{proof}
Since $\al$ is trace preserving, it is a well defined normal $*$-homomorphism. Assuming that $\sum_{g \in \Gamma} |\tau(\om(g,x))|^2 < +\infty$ for a.e.\ $x \in X$, we prove that the bimodule $\bim{\al(L(\Gamma))}{L^2(P \ovt B)}{P \ot 1}$ is coarse.

Denote by $\cC$ the set of Borel subsets $W \subset X$ with the property that
$$\int_W \sum_{g \in \Gamma} |\tau(\om(g,x))|^2 \; d\mu(x) < + \infty \; .$$
By assumption, $\cC$ contains Borel sets $W$ with $\mu(X \setminus W)$ being arbitrarily small. Therefore, the vectors $1_W \ot v$ with $W \in \cC$ and $v \in \cU(B)$ generate $L^2(P \ovt B)$ as an $L(\Gamma)$-$P$-bimodule.

It thus suffices to fix $W \in \cC$ and $v \in \cU(B)$ and prove that
$$\psi : L(\Gamma) \otalg P\op \to \C : \psi(a \ot d\op) = \langle \al(a) \cdot (1_W \ot v) \cdot (d \ot 1) , 1_W \ot v \rangle$$
extends to a normal functional on $L(\Gamma) \ovt P\op$. By construction, for all $g,h \in \Gamma$ and $F \in L^\infty(X)$, we have
\begin{align*}
\psi(u_g \ot (Fu_h)\op) &= \langle (u_g \ot 1) \, (W \ni x \mapsto F(x) \, \om(g,x) \, v) \, (u_h \ot 1) , 1_W \ot v \rangle \\
&= \delta_{g,h^{-1}} \, \int_{W \cap h \cdot W} F(x) \, \tau(\om(g,x)) \, d\mu(x) = \langle (u_g \ot 1) \, (1 \ot 1) \, (1 \ot F u_h) , \zeta \rangle \; ,
\end{align*}
where $\zeta \in L^2(L(\Gamma) \ovt P)$ is defined by
$$\zeta = \sum_{g \in \Gamma} u_g \ot 1_{W \cap g^{-1} \cdot W} \, \Om_g^* \, u_g^* \quad\text{with $\Om_g \in L^\infty(X)$ given by $\Om_g(x) = \tau(\om(g,x))$.}$$
Note that $\zeta$ is well defined because
$$\|\zeta\|_2^2 \leq \sum_{g \in \Gamma} \int_W |\Om_g|^2 \, d\mu < +\infty \quad\text{because $W \in \cC$.}$$
We have thus proven that $\psi(a \ot d\op) = \langle (a \ot 1) (1 \ot 1) (1 \ot d), \zeta \rangle$, so that $\psi$ indeed extends to a normal functional on $L(\Gamma) \ovt P\op$. This means that $\cH$ defines a measure equivalence embedding of $L(\Gamma)$ into $B$.

Now assume that we moreover have a finite subset $\cF \subset L^2(X,L^2(B))$ satisfying \eqref{eq.condition-ME}. Viewing $\cF \subset L^2(X,L^2(B)) \subset L^2(P \ovt B)$, it remains to prove that $\cF$ generates $\cH$ as a Hilbert $L(\Gamma)$-$P$-bimodule. Denote by $\cK$ the closed linear span of $\al(L(\Gamma)) \cdot \cF \cdot (P \ot 1)$. Since
$$\al(u_g) \cdot \vphi \cdot (u_g^* \ot 1) = \bigl( x \mapsto \om(g,g^{-1} \cdot x) \vphi(g^{-1} \cdot x) \bigr) \; ,$$
we conclude that $\cK_0 \subset \cK$, where $\cK_0 \subset L^2(X) \ovt L^2(B)$ is defined as the closed linear span of
$$\{ x \mapsto \om(g,g^{-1} \cdot x) \vphi(g^{-1} \cdot x) F(x) \mid g \in \Gamma , \vphi \in \cF , F \in L^\infty(X) \} \; .$$
It suffices to prove that $\cK_0 = L^2(X) \ovt L^2(B)$. Denote by $p$ the orthogonal projection of $L^2(X) \ovt L^2(B)$ onto $\cK_0$. Since $\cK_0$ is an $L^\infty(X)$-module, we have that $p = (x \mapsto p_x)$, where $p_x \in B(L^2(B))$ are orthogonal projections. Write $q_x = 1-p_x$. Then for all $g \in \Gamma$, $\vphi \in \cF$ and a.e.\ $x \in X$, we conclude that
$$q_x \bigl( \om(g,g^{-1} \cdot x) \vphi(g^{-1} \cdot x) \bigr) = 0 \; .$$
By our assumption, $q_x = 0$ for a.e.\ $x \in X$. So, $p=1$ and $\cK_0 = L^2(X) \ovt L^2(B)$.
\end{proof}

\subsection{Relation to measure equivalence of groups}

The following proposition provides further motivation for Definition \ref{def.ME-II1}. We prove that measure equivalence of discrete groups implies measure equivalence of the associated group von Neumann algebras. We do not expect that the converse implication holds, but this remains an open question. We actually expect that there should exist countable groups $\Gamma$ and $\Lambda$ with infinite conjugacy classes such that $L(\Gamma) \cong L(\Lambda)$ and such that $\Gamma$ and $\Lambda$ are not measure equivalent.

\begin{proposition}\label{prop.ME-groups}
Let $\Gamma$ and $\Lambda$ be countable groups. If $\Gamma$ admits a measure equivalence embedding into $\Lambda$, then the same holds for $L(\Gamma)$ and $L(\Lambda)$. If $\Gamma$ and $\Lambda$ are measure equivalent, then the same holds for $L(\Gamma)$ and $L(\Lambda)$.
\end{proposition}
\begin{proof}
First assume that $\Gamma$ admits a measure equivalence embedding into $\Lambda$. We can then take a $\si$-finite standard measure space $(Z,\mu)$ with free measure preserving commuting actions of $\Gamma$ (on the left) and of $\Lambda$ (on the right) that both admit a fundamental domain, say $Y \subset Z$ for the $\Gamma$-action and $X \subset Z$ for the $\Lambda$-action, such that $\mu(X) < +\infty$.

Note that we get a unique measure preserving action $\Gamma \actson (X,\mu)$ denoted by $g \ast x$ and $1$-cocycle $\om : \Gamma \times X \to \Lambda$ such that
$$g \cdot x = (g \ast x) \cdot \om(g,x) \quad\text{for all $g \in \Gamma$, $x \in X$.}$$
Note that $\om(g,x) = e$ iff $g \cdot x \in X$.

We view $\om$ as a $1$-cocycle taking values in the unitary group of $L(\Lambda)$, equipped with its canonical trace $\tau$. Then, $\tau(\om(g,x))$ only takes the values $0$ or $1$, depending on whether $\om(g,x) = e$ or $\om(g,x) \neq e$. By Proposition \ref{prop.source-ME}, it thus suffices to prove that for a.e.\ $x \in Z$,
$$\# \{g \in \Gamma \mid g \cdot x \in X \} <+\infty \; .$$
For every $g,h \in \Gamma$, we have that $h \cdot Y \cap g^{-1} \cdot X = g^{-1} \cdot (X \cap (gh) \cdot Y)$. So for every $h \in \Gamma$, we have that
$$\sum_{g \in \Gamma} \mu(h \cdot Y \cap g^{-1} \cdot X) = \sum_{g \in \Gamma} \mu(X \cap (gh) \cdot Y) = \mu(X) < +\infty \; .$$
By the Borel-Cantelli lemma, for every $h \in \Gamma$, the set $\{g \in \Gamma \mid g \cdot x \in X \}$ is finite for a.e.\ $x \in h \cdot Y$. Since this holds for all $h \in \Gamma$, we have proven that $\{g \in \Gamma \mid g \cdot x \in X \}$ is finite for a.e.\ $x \in Z$.

If $\Gamma$ and $\Lambda$ are measure equivalent, we may assume above that also $\mu(Y) < +\infty$ and that the action $\Gamma \times \Lambda \actson Z$ is ergodic. Since the action $\Lambda \actson \Gamma \backslash Z$ is ergodic and finite measure preserving, we can choose a finite set $\cF$ of Borel maps $\vphi : X \to \Lambda$ such that
\begin{equation}\label{eq.we-have-this-now}
\{ \Gamma \cdot x \cdot \vphi(x) \mid x \in X, \vphi \in \cF\} = Z \quad\text{up to measure zero.}
\end{equation}
By Proposition \ref{prop.source-ME}, it suffices to prove that for a.e.\ $x \in X$,
\begin{equation}\label{eq.nice-goal}
\Lambda = \{ \om(g,g^{-1} \ast x) \vphi(g^{-1} \ast x) \mid g \in \Gamma\} \; .
\end{equation}
But for all $\lambda \in \Lambda$ and a.e.\ $x \in X$, we find by \eqref{eq.we-have-this-now} an element $x' \in X$ and $\vphi \in \cF$ such that
$$x \cdot \lambda = g \cdot x' \cdot \vphi(x') \; .$$
This precisely means that $x' = g^{-1} \ast x$ and $\om(g,x') = \lambda \vphi(x')^{-1}$. Thus, $\om(g,g^{-1} \ast x) \vphi(g^{-1} \ast x) = \lambda$ and \eqref{eq.nice-goal} is proven.
\end{proof}

\subsection{Amenability, embeddings of free group factors, Haagerup property}

Theorem \ref{thm.main-ME-embedding-II-1-factors} stated in the introduction follows from the following result and Proposition \ref{prop.inherited} below.

\begin{theorem}\label{thm.free-group-factor-ME}
The free group factor $L(\F_2)$ admits a measure equivalence embedding into any nonamenable II$_1$ factor.
\end{theorem}
\begin{proof}
Let $M$ be a nonamenable II$_1$ factor. By \cite[Remark 5.29]{Con76}, we can choose $n \in \N$ and unitaries $u_1,\ldots,u_n \in \cU(M)$ such that
$$T := \frac{1}{n} \sum_{i=1}^n u_i \ot \overline{u_i} \in M \ovt M\op$$
satisfies $\|T\| < 1$. Taking a sufficiently large power of $(T+T^*)/2$, we may assume that $T = T^*$ and that $\|T\| < 1/3$.

Write $X_0 = \{1,\ldots,n\}$ and denote by $\mu_0$ the normalized counting measure on $X_0$. Write $\Gamma = \F_2$ with free generators $a,b \in \Gamma$ and consider the Bernoulli action $\Gamma \actson (X,\mu) = (X_0 \times X_0,\mu_0 \times \mu_0)^\Gamma$. Denote by $\om : \Gamma \times X \to \cU(M)$ the unique $1$-cocycle defined by
$$\om(a,x) = u_i \;\;\text{and}\;\;\om(b,x) = u_j \;\;\text{if}\;\; x_e = (i,j) \; .$$
By Proposition \ref{prop.source-ME}, it suffices to prove that
\begin{equation}\label{eq.we-need-this}
\sum_{g \in \Gamma} |\tau(\om(g,x))|^2 < + \infty \quad\text{for a.e.\ $x \in X$.}
\end{equation}
One computes that
$$\int_X (\om(g,x) \ot \overline{\om(g,x)}) \, d\mu(x) = T^{|g|} \quad\text{for all $g \in \Gamma$,}$$
where $g \mapsto |g|$ denotes the word length. Thus,
$$
\sum_{g \in \Gamma} \int_X |\tau(\om(g,x))|^2 \, d\mu(x) = \sum_{g \in \Gamma} (\tau \ot \tau)(T^{|g|})
\leq \sum_{g \in \Gamma} \|T\|^{|g|} < +\infty \; ,
$$
because $\|T\| < 1/3$ and because there are $4 \cdot 3^{n-1}$ elements of length $n$ in $\Gamma$. In particular, \eqref{eq.we-need-this} holds.
\end{proof}

\begin{proposition}\label{prop.inherited}
Let $A$ and $B$ be finite von Neumann algebras and assume that $A$ admits a measure equivalence embedding into $B$.
\begin{enumlist}
\item If $B$ is amenable, also $A$ is amenable.
\item If $B$ has the Haagerup approximation property, also $A$ has the Haagerup approximation property.
\end{enumlist}
\end{proposition}

\begin{proof}
Choose a finite von Neumann algebra $P$, a projection $p \in P \ovt B$ and a faithful normal embedding $\al : A \to p(P \ovt B)p$ such that the bimodule $\bim{\al(A)}{p L^2(P \ovt B)}{P \ot 1}$ is coarse.

1.\ Assume that $B$ is amenable. Then, the trivial $(P \ovt B)$-bimodule is weakly contained in $L^2(P \ovt B) \ot_{P \ot 1} L^2(P \ovt B)$. A fortiori, the trivial $A$-bimodule is weakly contained in
$$\bim{\al(A)}{\bigl(p L^2(P \ovt B) \ot_{P \ot 1} L^2(P \ovt B) p\bigr)}{\al(A)} \; .$$
But this bimodule is coarse. So, $A$ is amenable.

2.\ Assume that $B$ has the Haagerup approximation property. Fix faithful normal tracial states $\tau$ on $P$ and $B$. Equip $A$ with the (non-normalized) faithful trace $\tau(a) = (\tau \ot \tau)\al(a)$ for all $a \in A$. Since $B$ has the Haagerup approximation property, by \cite[Propositions 2.2 and 2.4]{Jol00}, we can choose a sequence of normal, completely positive maps $\vphi_n : B \to B$ such that $\vphi_n(1) \leq 1$ and $\tau \circ \vphi_n \leq \tau$ for all $n$, such that $\|\vphi_n(b) - b\|_2 \to 0$ for every $b \in B$ and such that the induced bounded operators $T_n : L^2(B) \to L^2(B) : T_n(b) = \vphi_n(b)$ are compact.

We now use the method and ideas of \cite[Theorem 5.9]{DKP22}. Given a tracial von Neumann algebra $(M,\tau)$, we say that a bounded operator $S : L^2(A) \to L^2(M)$ is \emph{almost compact} if for every $\eps > 0$, there exist projections $p \in A$ and $q \in M$ with $\tau(1-p) < \eps$, $\tau(1-q) < \eps$ such that $a \mapsto q S(pap) q$ is compact as an operator from $L^2(A)$ to $L^2(M)$. One then checks that the almost compact operators from $L^2(A)$ to $L^2(M)$ form a norm closed linear subspace of $B(L^2(A),L^2(M))$.

We claim that for every compact operator $T : L^2(B) \to L^2(B)$, the operator
$$L^2(A) \to L^2(P \ovt B) : a \mapsto (1 \ot T)(\al(a))$$
is almost compact. We prove below that for every $b \in B$, the operator
$$T_b : L^2(A) \to L^2(P) : T_b(a) = (\id \ot \tau)(\al(a) (1 \ot b))$$
is almost compact. It then follows that for every rank one operator $T : L^2(B) \to L^2(B)$, the operator $(1 \ot T) \circ \al$ is almost compact, even with only using projections from $A$ and $P$. So, then also the claim follows.

Note that $\langle T_b(a) , d \rangle = \langle \al(a) \cdot (p (1 \ot b)) \cdot (d^* \ot 1) , p \rangle$ for all $a \in A$, $d \in P$. Since $\eta = p (1 \ot b)$ and $\xi = p$ are bounded vectors in the coarse $A$-$P$-bimodule $\bim{\al(A)}{p L^2(P \ot B)}{P \ot 1}$, it follows from Lemma \ref{lem.coarse-almost-compact} below that $T_b$ is almost compact.

Having proven the claim, we can choose projections $p_n \in A$ and $q_n \in P \ovt B$ such that $R_n : L^2(A) \to L^2(P \ovt B) : R_n(a) = q_n (\id \ot \vphi_n)(\al(p_n a p_n)) q_n$ is compact for every $n$ and $\|R_n(a) - \al(a)\|_2 \to 0$ for all $a \in A$. We define $\psi_n : A \to A : \psi_n = \al^{-1} \circ E_{\al(A)} \circ R_n$. Then $\psi_n$ is a sequence of subunital, subtracial completely positive maps, converging pointwise to the identity and inducing compact operators on $L^2(A)$. Thus, $A$ has the Haagerup property.
\end{proof}

Recall from the proof of Proposition \ref{prop.inherited} the concept of an almost compact operator $L^2(A) \to L^2(P)$, given tracial von Neumann algebras $A$ and $P$. Let $\bim{A}{\cH}{P}$ be a Hilbert $A$-$P$-bimodule. Recall that a vector $\xi \in \cH$ is said to be left bounded if the operator $L_\xi : L^2(P) \to \cH : b \mapsto \xi \cdot b$ is well defined and bounded. We similarly consider right bounded vectors $\eta \in \cH$ and their operators $R_\eta : L^2(A) \to \cH : a \mapsto a \cdot \eta$.

The following lemma is part of \cite[Theorem 5.10]{DKP22}. For completeness, we give a proof in our more specific setting.

\begin{lemma}\label{lem.coarse-almost-compact}
Let $A$ and $P$ be tracial von Neumann algebras and let $\bim{A}{\cH}{P}$ be a Hilbert $A$-$P$-bimodule. Let $\cH_0 \subset \cH$ be a set of left bounded vectors and let $\cH_1 \subset \cH$ be a set of right bounded vectors such that $L_\xi^* R_\eta$ is compact for all $\xi \in \cH_0$, $\eta \in \cH_1$. Assume that for $i \in \{0,1\}$, the linear span of $A \cdot \cH_i \cdot P$ is dense in $\cH$.

Then for every left bounded vector $\xi \in \cH$ and right bounded vector $\eta \in \cH$, the operator $L_\xi^* R_\eta : L^2(A) \to L^2(P)$ is almost compact.

In particular, this conclusion holds if $\bim{A}{\cH}{P}$ is a coarse $A$-$P$-bimodule.
\end{lemma}
\begin{proof}
Denote by $\cK_i$ the linear span of $A \cdot \cH_i \cdot P$. Since
$$L_{a \xi b}^* R_{c \eta d} = \lambda_P(b)^* \, \rho_P(d) \, L_\xi^* R_\eta \, \lambda_A(a)^* \, \rho_A(c) \; ,$$
we get that $L_\xi^* R_\eta$ is compact for all $\xi \in \cK_0$ and $\eta \in \cK_1$.

Let $\xi \in \cH$ be an arbitrary left bounded vector and choose a sequence $\xi_n \in \cK_0$ such that $\|\xi - \xi_n\| \to 0$. Write $\zeta_n = \xi - \xi_n$. Then, $S_n := L_{\zeta_n}^* L_{\zeta_n}$ is a sequence in $P^+$ satisfying $\|S_n\|_1 \to 0$. After passage to a subsequence, we may assume that there is an increasing sequence of projections $q_k \in P$ such that $q_k \to 1$ strongly and, for every fixed $k$, $\lim_n \|S_n q_k\| = 0$. In particular, $\lim_n \|L_{\xi_n} q_k - L_{\xi} q_k\| = 0$ for every fixed $k$.

Given a right bounded vector $\eta \in \cH$, we similarly find $\eta_n \in \cK_1$ and $p_k \in A$ such that $\lim_n \|p_k R_{\eta_n} - p_k R_\eta\| = 0$ for every fixed $k$. Writing $T = L_\xi^* R_\eta$ and $T_n = L_{\xi_n}^* R_{\eta_n}$, it follows that for every $k$, the operator $a \mapsto q_k T(p_k a p_k) q_k$ is compact, because it is the norm limit of the compact operators $a \mapsto q_k T_n(p_k a p_k) q_k$. So, $T$ is almost compact.
\end{proof}

\subsection{Property (T)}

Property~(T) for von Neumann algebras was introduced in \cite{Con80}. We use the following equivalent definition, see e.g.\ \cite[Proposition 4.1]{Pop02}: a tracial von Neumann algebra $(M,\tau)$ has property~(T) iff for any $M$-$M$-bimodule $\bim{M}{\cH}{M}$ and sequence $\xi_n \in \cH$ satisfying
$$\|\xi_n \cdot a \| \leq \|a\|_2 \;\; , \;\; \|a \cdot \xi_n\| \leq \|a\|_2 \;\; , \;\; \lim_n \|a \cdot \xi_n - \xi_n \cdot a\| = 0 \quad\text{for all $a \in M$,}$$
we have that $\sup \{\|a \cdot \xi_n - \xi_n \cdot a\|\mid a \in M , \|a\|\leq 1\} \to 0$. Moreover, this property does not depend on the choice of trace $\tau$.

\begin{proposition}\label{prop.preserve-T}
Property (T) is preserved under measure equivalence of finite von Neumann algebras.
\end{proposition}
\begin{proof}
Let $A$ and $B$ be finite von Neumann algebras that are measure equivalent. By Lemma \ref{lem.charac}, we can fix a finite von Neumann algebra $P$, projections $p \in P \ovt B$ and $q \in P\op \ovt A$, faithful normal $*$-homomorphisms $\al : A \to p(P \ovt B)p$ and $\be : B \to q(P\op \ovt A)q$ and a unitary operator
\begin{equation}\label{eq.unitary-equiv}
U : p L^2(P \ovt B) \to L^2(P\op \ovt A)q \quad\text{s.t.}\quad U(\al(a) \cdot \eta \cdot (d \ot b)) = (d\op \ot a) \cdot U(\eta) \cdot \be(b)
\end{equation}
for all $a \in A$, $b \in B$ and $d \in P$. Fix faithful normal tracial states $\tau$ on $A$ and $P$. We equip $B$ with the non-normalized faithful trace $\tau(b) = (\tau \ot \tau)\be(b)$. We may view the $L^2$-spaces in \eqref{eq.unitary-equiv} as taken w.r.t.\ these traces.

Assume that $A$ has property~(T). We prove that $B$ has property~(T). So we fix a Hilbert $B$-$B$-bimodule $\cH$ with a sequence of vectors $\eta_n \in \cH$ satisfying
$$\|b \cdot \eta_n \| \leq \|b\|_2 \;\; , \;\; \|\eta_n \cdot b \| \leq \|b\|_2 \;\;\text{and}\;\; \lim_n \|b \cdot \eta_n - \eta_n \cdot b\| = 0 \;\;\text{for all $b \in B$.}$$
We have to prove that $\sup \bigl\{ \|b \cdot \eta_n - \eta_n \cdot b\| \bigm| b \in B , \|b\| \leq 1 \bigr\} \to 0$.

Consider the Hilbert $A$-$A$-bimodule $\cK = p \cdot (L^2(P) \ot \cH) \cdot p$ with bimodule structure given by $\al(a) \cdot \eta \cdot \al(a')$. Since $A$ has property~(T), the sequence of almost central vectors $p \cdot (1 \ot \eta_n) \cdot p \in \cK$ satisfies
$$\sup \bigl\{ \|\al(a) \cdot (1 \ot \eta_n) \cdot \al(a^*) - p \cdot (1 \ot \eta_n) \cdot p \|  \bigm| a \in \cU(A) \bigr\} \to 0 \; .$$
Defining $\rho_n$ as the unique element of minimal norm in the closed convex hull of
$$\{ \al(a) \cdot (1 \ot \eta_n) \cdot \al(a^*) \mid a \in \cU(A)\}$$
inside $\cK$, we get that $\al(a) \cdot \rho_n = \rho_n \cdot \al(a)$ for all $a \in A$. Since $\|(1 \ot \eta_n) \cdot p - p \cdot (1 \ot \eta_n)\| \to 0$, we also get that $\|\rho_n - (1 \ot \eta_n) \cdot p\| \to 0$.

Define $\xi \in p L^2(P \ovt B)$ by $\xi = U^*(q)$. For all $b,b' \in B$, we have that
$$\langle \xi \cdot (1 \ot b), \xi \cdot (1 \ot b') \rangle = \langle \be(b) , \be(b') \rangle = \langle b , b' \rangle \; .$$
So, we can view $\xi \in L^2(P) \ovt B$ in such a way that $\xi^* \xi = 1$ in $B$. Denoting by $\lambda_\cH : B \to B(\cH)$ the left module action, $S_\xi := (\id \ot \lambda_\cH)(\xi)$ is then a well-defined isometry from $\cH$ to $p \cdot (L^2(P) \ot \cH)$. By definition, $S_\xi(\eta \cdot b) = (S_\xi(\eta)) \cdot (1 \ot b)$ for all $\eta \in \cH$ and $b \in B$. Also, whenever $\eta \in \cH$ is a right bounded vector, $S_\xi(\eta) = (1 \ot R_\eta)(\xi)$, where $R_\eta : L^2(B) \to \cH$ is given by $R_\eta(b) = b \cdot \eta$.

Since $S_\xi$ is an isometry, it suffices to prove that $\|S_\xi(\eta_n \cdot b) - S_\xi(b \cdot \eta_n)\| \to 0$ uniformly on the unit ball $(B)_1$ of $B$.

For all $a \in \cU(A)$ and all $b \in P \ovt B$, we have that
$$\|\al(a) \cdot (1 \ot \eta_n) \cdot (\al(a^*) b)\| \leq \|(1 \ot \eta_n) \cdot (\al(a^*) b)\| \leq \|\al(a^*) b\|_2 \leq \|b\|_2 \; .$$
By taking convex combinations, it follows that $\|\rho_n \cdot b\| \leq \|b\|_2$. We thus find a well-defined bounded operator $L_{\rho_n} : L^2(P \ovt B) \to L^2(P) \ot \cH$ satisfying $L_{\rho_n}(b) = \rho_n \cdot b$ for all $b \in P \ovt B$. Also note that $\|L_{\rho_n}\| \leq 1$ for all $n$.

We claim that $\|L_{\rho_n}(\zeta) - (1 \ot R_{\eta_n})(\zeta)\| \to 0$ for all $\zeta \in p L^2(P \ovt B)$. Since also $\|R_{\eta_n}\| \leq 1$ for all $n$, it suffices to prove this claim for $\zeta = b \in p(P \ovt B)$. In that case, $L_{\rho_n}(b) = \rho_n \cdot b$, while $(1 \ot R_{\eta_n})(b) = b \cdot (1 \ot \eta_n)$. The claim then follows because $\|b \cdot (1 \ot \eta_n) - (1 \ot \eta_n) \cdot b\| \to 0$ and $\|(1 \ot \eta_n) \cdot p - \rho_n\| \to 0$.

Note that $S_\xi(\eta_n \cdot b) = (S_\xi(\eta_n)) \cdot (1 \ot b)$ for all $b \in B$, while $S_\xi(\eta_n) = (1 \ot R_{\eta_n})(\xi)$. Using the claim in the previous paragraph, it follows that
\begin{equation}\label{eq.stap1}
\|S_\xi(\eta_n \cdot b) - (L_{\rho_n}(\xi)) \cdot (1 \ot b)\| \to 0 \quad\text{uniformly on $(B)_1$.}
\end{equation}

Denote by $\pi_1 : P\op \ovt A \to B(p L^2(P \ovt B))$ the normal $*$-homomorphism given by $\pi_1(d\op \ot a)(\zeta) = \al(a) \zeta (d \ot 1)$. This is well-defined, because $\pi_1(x)(U^*(\gamma)) = U^*(x \cdot \gamma)$. For the same reason, we have that $\pi_1(\be(b))(\xi) = \xi \cdot (1 \ot b)$ for all $b \in B$. Note that the range of $\pi_1$ is contained in $p (B(L^2(P)) \ovt B) p$. Then also $\pi_2 = (\id \ot \lambda_\cH) \circ \pi_1$ is a well-defined normal $*$-homomorphism from $P\op \ovt A$ to $B(p \cdot (L^2(P) \ot \cH))$. We have that $\pi_2(d\op \ot a)(\zeta) = \al(a) \cdot \zeta \cdot (d \ot 1)$ whenever $\zeta \in p \cdot (L^2(P) \ot \cH)$.

With these notations, we first get that $(L_{\rho_n}(\xi)) \cdot (1 \ot b) = L_{\rho_n}(\xi \cdot (1 \ot b)) = L_{\rho_n}(\pi_1(\be(b))(\xi))$. Next, we note that for all $d \in P$, $a \in A$ and $x \in p(P \ovt B)$,
\begin{align*}
L_{\rho_n}(\pi_1(d\op \ot a)(x)) &= \rho_n \cdot (\al(a) x (d \ot 1)) = \al(a) \cdot \rho_n \cdot (x (d \ot 1)) \\ &= \al(a) \cdot L_{\rho_n}(x) \cdot (d \ot 1) = \pi_2(d\op \ot a)(L_{\rho_n}(x)) \; .
\end{align*}
By density and normality, $L_{\rho_n}(\pi_1(c)(\zeta)) = \pi_2(c)(L_{\rho_n}(\zeta))$ for all $c \in P\op \ovt A$ and $\zeta \in p L^2(P \ovt B)$. In particular,
$$L_{\rho_n}(\pi_1(\be(b))(\xi)) = \pi_2(\be(b))(L_{\rho_n}(\xi)) \quad\text{for all $b \in B$.}$$
Since $\|L_{\rho_n}(\xi) - (1 \ot R_{\eta_n})(\xi)\| \to 0$, we conclude that
\begin{equation}\label{eq.stap2}
\|(L_{\rho_n}(\xi)) \cdot (1 \ot b) - \pi_2(\be(b))(1 \ot R_{\eta_n})(\xi)\| \to 0 \quad\text{uniformly on $(B)_1$.}
\end{equation}
Now,
\begin{align*}
\pi_2(\be(b))(1 \ot R_{\eta_n})(\xi) &= (1 \ot R_{\eta_n})(\pi_1(\be(b))(\xi)) = (1 \ot R_{\eta_n})(\xi \cdot (1 \ot b)) \\
& = (1 \ot R_{b \cdot \eta_n})(\xi) = S_\xi(b \cdot \eta_n) \; .
\end{align*}
Combining this equality with \eqref{eq.stap1} and \eqref{eq.stap2}, we find that $\|S_\xi(\eta_n \cdot b) - S_\xi(b \cdot \eta_n)\| \to 0$ uniformly on the unit ball $(B)_1$.
\end{proof}

\subsection{Example}

In Theorem \ref{thm.free-group-factor-ME}, we proved that the free group factor $L(\F_2)$ admits a measure equivalence embedding into any nonamenable II$_1$ factor $M$. It remains an open problem to decide if $L(\F_2)$ actually embeds into any nonamenable II$_1$ factor.

In the following example, we show that the existence of a measure equivalence embedding of $A$ into $B$ is, in general, a strictly weaker property than the existence of an embedding of $A$ into $B$, or into an amplification $B^t$. So there is certainly no hope to derive in an abstract way from Theorem \ref{thm.free-group-factor-ME} that $L(\F_2)$ embeds into every nonamenable II$_1$ factor.

To give such an example, we necessarily have to rely on sophisticated results. Until the paper \cite{PV21}, basically all non-embeddability results for II$_1$ factors were based on qualitative properties of the style: II$_1$ factors with a certain property (e.g.\ property~(T)) do not embed into a II$_1$ factor with a ``strongly opposite'' property (e.g.\ Haagerup property). As we have shown above, such qualitative obstructions for the existence of an embedding, typically also are an obstruction for the existence of a measure equivalence embedding. On the other hand, \cite[Theorem A]{PV21} provides more quantitative non-embeddability results that we can use to prove the following result.

Recall that for a finite dimensional C$^*$-algebra $A_0$, the \emph{Markov trace} $\tau(a)$, $a \in A_0$, is defined as the trace of the left multiplication operator $L_a : A_0 \to A_0$. Dividing by $\dim A_0$, we obtain the \emph{normalized Markov trace}.

\begin{proposition}\label{prop.counterexample}
For every finite-dimensional C$^*$-algebra $A_0$, denote by $\tau_{A_0}$ the normalized Markov trace on $A_0$. Take $\Gamma = \F_2$ and define the II$_1$ factor $M(A_0)$ as the left-right Bernoulli crossed product
$$M(A_0) = (A_0,\tau_{A_0})^{\ot \Gamma} \rtimes (\Gamma \times \Gamma) \; .$$
Let $A_0$, $A_1$ be finite-dimensional C$^*$-algebras of the same dimension.
\begin{enumlist}
\item The II$_1$ factors $M(A_0)$ and $M(A_1)$ are measure equivalent.
\item If $A_0 \not\cong A_1$, there does not exist an embedding of $M(A_0)$ into an amplification of $M(A_1)$.
\end{enumlist}
\end{proposition}

\begin{proof}
1.\ Let $A_0$ be any finite-dimensional C$^*$-algebra and denote by $A_1 = \C^{\dim A_0}$ the unique abelian C$^*$-algebra having the same dimension. Since, by Lemma \ref{lem.transitive}, measure equivalence is an equivalence relation, it suffices to prove that $M(A_0)$ and $M(A_1)$ are measure equivalent.

Write $A_0 = \bigoplus_{s=1}^k M_{n_s}(\C)$. For every $s$, choose a primitive $n_s$-th root of unity $\om_s$. Define the commuting automorphisms $\al_s, \be_s \in \Aut(M_{n_s}(\C))$ of order $n_s$ given by
$$(\al_s(a))_{ij} = \om_s^{i-j} \, a_{ij} \quad\text{and}\quad (\be_s(a))_{ij} = a_{i+1,j+1} \;\;\text{identifying $n_s+1 = 1$.}$$
Define
$$\theta_s : M_{n_s}(\C) \to M_{n_s}(\C) \ot \ell^\infty\bigl(\{1,\ldots,n_s\} \times \{1,\ldots,n_s\}\bigr) : \theta_s(a) = \sum_{k,l=1}^{n_s} (\al_s^k \be_s^l)(a) \ot \delta(k,l) \; .$$
Taking the direct sum over $s$, we find the $*$-homomorphism $\theta : A_0 \to A_0 \ot A_1$ satisfying $(\id \ot \tau_{A_1})(\theta(a)) = \tau_{A_0}(a) \, 1$ for all $a \in A_0$. Taking a tensor product over $\Gamma$ and the crossed product, we find the natural embedding $\al : M(A_0) \to M(A_0) \ovt M(A_1)$ satisfying $(\id \ot \tau)\al(a) = \tau(a)$ for all $a \in M(A_0)$.

Since $(\id \ot \tau_{A_1})(\theta(a)) = \tau_{A_0}(a) \, 1$ for all $a \in A_0$, the linear map $A_0 \ot A_0 \to A_0 \ot A_1 : a \ot b \mapsto \theta(a)(b \ot 1)$ is isometric, and hence bijective because the dimensions are finite. It then follows that the linear span of $\al(M(A_0)) (M(A_0) \ot 1)$ is $\|\,\cdot\,\|_2$-dense in $M(A_0) \ovt M(A_1)$. By Lemma \ref{lem.criterion-embedding} below, it follows that $M(A_0)$ and $M(A_1)$ are measure equivalent.

2.\ By \cite[Theorem A]{PV21}, if $M(A_0)$ embeds into $M(A_1)^t$ for some $t > 0$, there must exist a trace preserving embedding of $(A_0,\tau_{A_0})$ into $(A_1,\tau_{A_1})$. Since $A_0$ and $A_1$ have the same finite dimension, this forces $A_0 \cong A_1$.
\end{proof}

\begin{lemma}\label{lem.criterion-embedding}
Let $(A,\tau)$ and $(B,\tau)$ be tracial von Neumann algebras. Assume that $(P,\tau)$ is a tracial von Neumann algebra and $\al : A \to P \ovt B$ is a unital $*$-homomorphism satisfying $(\id \ot \tau)\al(a) = \tau(a) \, 1$ for all $a \in A$. Then, $\cH = \bim{\al(A)}{L^2(P \ovt B)}{1 \ot B}$ is a measure equivalence embedding of $A$ into $B$.

If moreover $\al(A)(P \ot 1)$ has a dense linear span in $L^2(P \ovt B)$, then $\cH$ is a measure equivalence of $A$ and $B$.
\end{lemma}
\begin{proof}
We only have to show that the bimodule $\bim{\al(A)}{L^2(P \ovt B)}{P \ot 1}$ is coarse. Choose a faithful normal state $\Omega$ on $B(L^2(P))$. Define the $*$-homomorphism
$$\psi : A \otalg P\op \to B(L^2(P)) \ovt B : \psi(a \ot d\op) \zeta = \al(a)\zeta (d \ot 1) \; .$$
Then, $(\Omega \ot \tau)\psi(x) = (\tau \ot \om_0)(x)$ for all $x \in A \otalg P\op$, where $\om_0$ is a faithful normal state on $P\op$. So, $\psi$ extends to a normal $*$-homomorphism from $A \ovt P\op$ to $B(L^2(P)) \ovt B$. In particular, $\bim{\al(A)}{L^2(P \ovt B)}{P \ot 1}$ is coarse.
\end{proof}

\end{document}